\documentclass[12pt,a4paper,reqno]{amsart}
\usepackage[all,poly,color]{xy}
\usepackage{amsmath,amssymb, amsthm, amsfonts}
\usepackage{color}
\usepackage[pagebackref,colorlinks]{hyperref}
\usepackage{enumerate}
\usepackage{comment}
\usepackage[dvipsnames]{xcolor}

\theoremstyle{plain}
\newtheorem {lemma}{Lemma}
\newtheorem {proposition}[lemma]{Proposition}
\newtheorem {theorem}[lemma]{Theorem}
\newtheorem {corollary}[lemma]{Corollary}

\theoremstyle{definition}
\newtheorem {definition}[lemma]{Definition}
\newtheorem {remark}[lemma]{Remark}
\newtheorem {example}[lemma]{Example}
\newtheorem {question}[lemma]{Question}

\newcommand{\N}{\mathbb{N}}

\newcommand{\C}{\mathcal{C}}
\newcommand{\End}{\operatorname{End}}
\newcommand{\Ann}{\operatorname{ann}}

\newcommand{\id}{\operatorname{id}}

\newcommand{\Walk}{\operatorname{Walk}}
\newcommand{\Ob}{\operatorname{Ob}}
\newcommand{\ob}{\operatorname{ob}}

\newcommand{\G}{\mathcal{G}}
\newcommand{\KP}{\operatorname{KP}}

\DeclareMathOperator{\RG}{\mathbf{RG}}
\DeclareMathOperator{\Modd}{\mathbf{Mod}}

\makeatletter
\@namedef{subjclassname@2020}{%
  \textup{2020} Mathematics Subject Classification}
\makeatother

\begin{document}

\title[Simple modules for Kumjian-Pask algebras]{Simple modules for Kumjian-Pask algebras}

\author{Raimund Preusser}
\address{Chebyshev Laboratory, St. Petersburg State University, Russia} \email{raimund.preusser@gmx.de}

\thanks{The work is supported by the Russian Science Foundation grant 19-71-30002}
\maketitle

\begin{abstract}
The paper introduces the notion of a representation $k$-graph $(\Delta,\alpha)$ for a given $k$-graph $\Lambda$. It is shown that any representation $k$-graph for $\Lambda$ yields a module for the Kumjian-Pask algebra $\KP(\Lambda)$, and the representation $k$-graphs yielding simple modules are characterised. Moreover, the category $\RG(\Lambda)$ of representation $k$-graphs for $\Lambda$ is investigated using the covering theory of higher-rank graphs.

\end{abstract}


\section{Introduction}
In a series of papers~\cite{vitt56,vitt57,vitt62}, William Leavitt studied algebras that are now denoted by $L(n,n+k)$ and have been coined Leavitt algebras. Leavitt path algebras $L(E)$, introduced in~\cite{AA,AMP}, are algebras associated to directed graphs $E$. For the graph $E$ with one vertex and $k+1$ loops, one recovers the Leavitt algebra $L(1,k+1)$. The Leavitt path algebras turned out to be a very rich and interesting class of algebras, whose studies so far constitute over 150 research papers. A comprehensive treatment of the subject can be found in the book~\cite{AASbook}. 

There have been a substantial number of papers devoted to (simple) modules over Leavitt path algebras. Ara and Brustenga~\cite{AB1,AB} studied their finitely presented modules, proving that the category of finitely presented modules over a Leavitt path algebra $L(E)$ is equivalent to a quotient category of the corresponding category of modules over the path algebra $KE$. A similar statement for graded modules over a Leavitt path algebra was established by Paul Smith~\cite{smith2}. 
Gon\c{c}alves and Royer~\cite{GR} obtained modules for Leavitt path algebras by introducing the notion of a branching system for a graph. Chen~\cite{C} used infinite paths in $E$ to obtain simple modules for the Leavitt path algebra $L(E)$. Numerous work followed, noteworthy the work of Ara-Rangaswamy and Rangaswamy~\cite{ARa,R-2,R-3} producing new simple modules associated to infinite emitters and characterising those algebras which have countably (finitely) many distinct isomorphism classes of simple modules. Abrams, Mantese and Tonolo~\cite{AMT} studied the projective resolutions for these simple modules. The recent work of \'Anh and Nam \cite{nam} provides another way to describe the so-called Chen and Rangaswamy simple modules. 

The Leavitt algebras $L(n,n+k)$ where $n>1$ can not be obtained via Leavitt path algebras. For this reason weighted Leavitt path algebras were introduced in~\cite{H-1}. For the weighted graph with one vertex and $n+k$ loops of weight $n$ one recovers the Leavitt algebra $L(n,n+k)$. If all the weights are $1$, then the weighted Leavitt path algebras reduce to the usual Leavitt path algebras. In a recent preprint \cite{HPS} the authors obtained modules for weighted Leavitt path algebras by introducing the notion of a representation graph for a weighted graph. They proved that each connected component $C$ of the category $\RG(E)$ of representation graphs for a weighted graph $E$ contains a universal object $T_C$, yielding an indecomposable $L_K(E)$-module $V_{T_C}$, and a unique object $S_C$ yielding a simple $L_K(E)$-module $V_{S_C}$. It was also shown that specialising to unweighted graphs, one recovers the simple modules of the usual Leavitt path algebras constructed by Chen via infinite paths.

Kumjian-Pask algebras $\KP(\Lambda)$, which are algebras associated to higher-rank graphs $\Lambda$, were introduced by Aranda Pino, Clark, an Huef and Raeburn \cite{pchr} and generalise the Leavitt path algebras. The definition was inspired by the higher-rank graph $C^*$-algebras introduced by Kumjian and Pask \cite{kp2000}. In \cite{pchr} the authors obtained modules for the Kumjian-Pask algebras using infinite paths and provided a necessary and sufficient criterion for the faithfulness of these modules. Kashoul-Radjabzadeh, Larki and Aminpour~\cite{KRLA} characterised primitive Kumjian-Pask algebras in graph-theoretic terms.

In the present paper we apply ideas from \cite{HPS} in order to obtain modules for Kumjian-Pask algebras. We introduce the notion of a representation $k$-graph $(\Delta,\alpha)$ for a given $k$-graph $\Lambda$.  We show that any representation $k$-graph $(\Delta,\alpha)$ for $\Lambda$ yields a module $V_{(\Delta,\alpha)}$ for the Kumjian-Pask algebra $\KP(\Lambda)$ and characterise the representation $k$-graphs yielding simple modules. Moreover, we investigate the category $\RG(\Lambda)$ of representation $k$-graphs for $\Lambda$ using the covering theory of higher-rank graphs developed in \cite{PQR}.

In Section 2 we recall some of the definitions and results of \cite{PQR}. In Section 3 we introduce the main notion of this paper, namely the notion of a representation $k$-graph. We show that each connected component $C$ of the category $\RG(\Lambda)$ contains objects $(\Omega_C,\zeta_C)$ and $(\Gamma_C,\xi_C)$ such that each object of $C$ is a quotient of $(\Omega_C,\zeta_C)$ and a covering of $(\Gamma_C,\xi_C)$. In Section 4 we recall the definition of a Kumjian-Pask algebra and define the $\KP(\Lambda)$-module $V_{(\Delta,\alpha)}$ associated to a representation $k$-graph $(\Delta,\alpha)$ for $\Lambda$. We show that (up to isomorphism) the representation $k$-graphs $(\Gamma_C,\xi_C)$ are precisely those representation $k$-graphs for $\Lambda$ that yield simple $\KP(\Lambda)$-modules. Moreover, we prove that $V_{(\Gamma_C,\xi_C)}\not\cong V_{(\Gamma_D,\xi_D)}$ if the connected components $C$ and $D$ of $\RG(\Lambda)$ are distinct. In Section 5 we obtain a necessary and sufficient criterion for the indecomposability of the modules $V_{(\Omega_C,\zeta_C)}$. We conclude that the modules $V_{(\Omega_C,\zeta_C)}$ are indecomposable if $k=1$. Section 6 contains a couple of examples.

Throughout the paper $K$ denotes a field and $K^{\times}$ the set of all nonzero elements of $K$. By a $K$-algebra we mean an associative (but not necessarily commutative or unital) $K$-algebra. The set of all nonnegative integers is denoted by $\mathbb N$.

\section{Coverings of higher-rank graphs}
\subsection{$k$-graphs}
For a positive integer $k$, we view the additive monoid $\mathbb{N}^k$ as a category with one object. A {\it $k$-graph} is a small category $\Lambda = (\Lambda^{\ob},\Lambda, r, s)$ together with a functor $d : \Lambda\rightarrow \mathbb{N}^k$, called the {\it degree map}, satisfying the following {\it factorisation property}:
if $\lambda\in \Lambda$ and $d(\lambda) = m+n$ for some $m, n\in \N^k$, then there are unique $\mu, \nu \in\Lambda$ such that $d(\mu) = m$, $d(\nu) = n$ and $\lambda= \mu\circ \nu$. An element $v\in\Lambda^{\ob}$ is called a {\it vertex} and an element $\lambda\in\Lambda$ a {\it path of degree $d(\lambda)$ from $s(\lambda)$ to $r(\lambda)$}.

For $u, v\in \Lambda^{\ob}$ we set $u\Lambda := r^{-1}(u)$, $\Lambda v := s^{-1}(v)$ and $u\Lambda v = u\Lambda \cap \Lambda v$. For $u, v\in \Lambda^{\ob}$ and $n\in \N^k$ we set $\Lambda^n:= d
^{-1}(n)$, $u\Lambda^n := u\Lambda \cap \Lambda^n$ and $\Lambda^nv := \Lambda v \cap \Lambda^n$. A {\it morphism} between $k$-graphs is a degree-preserving functor.

A $k$-graph $\Lambda$ is called {\it nonempty} if $\Lambda^{\ob}\neq \emptyset$, and {\it connected} if
the equivalence relation on $\Lambda^{\ob}$ generated by $\{(u, v) | u\Lambda v\neq \emptyset \}$ is $\Lambda^{\ob} \times \Lambda^{\ob}$.  $\Lambda$ is called {\it row-finite} if $v\Lambda^n$ is finite for any $v\in \Lambda^{\ob}$ and $n\in \mathbb{N}^k$. $\Lambda$ {\it has
no sources} if $v\Lambda^n$ is nonempty for any $v\in \Lambda^{\ob}$ and $n\in \mathbb{N}^k$. In this paper all $k$-graphs are assumed to be nonemtpy, connected, row-finite and to have no sources.

\subsection{The fundamental groupoid of a $k$-graph}

Recall that a {\it (directed) graph} $E$ is a tuple $(E^{0}, E^{1}, r,s)$, where $E^{0}$ and $E^{1}$ are sets and $r,s$ are maps from $E^1$ to $E^{0}$. We may think of each $e \in E^1$ as an edge pointing from the vertex $s(e)$ to the vertex $r(e)$. A {\it path} $p$ in a graph $E$ is a finite sequence $p=e_{n} \cdots e_{2}e_{1}$ of edges $e_{i}$ in $E$ such that $s(e_{i})=r(e_{i-1})$ for $2\leq i\leq n$. We define $r(p) =r(e_{n})$ and $s(p) = s(e_{1})$. The paths $p=e_{l}e_{l-1} \cdots e_{m+1}e_{m}$ where $1\leq m\leq l\leq n$ are called {\it subpaths} of $p$.

Let $\mathcal{C}$ be a category. The {\it underlying graph} of $\mathcal{C}$ is the graph $E(\mathcal{C})$ whose vertices are the objects of $\mathcal{C}$ and whose edges are the morphisms of $\C$ (the source and the range map are defined in the obvious way). Let $E(\mathcal{C})_d$ be the graph obtained from $E(\mathcal{C})$ by adding for any edge $e$ which is not an identity morphism in $\mathcal{C}$, an edge $e^*$ with reversed direction. A {\it walk} in $\mathcal{C}$ is a path $p$ in the graph $E(\mathcal{C})_d$. We denote by $\Walk(\mathcal{C})$ the set of all walks in $\mathcal{C}$. Moreover, we denote by $\Walk_u(\mathcal{C})$ the set of all walks starting in $u$, by ${}_v\!\Walk(\mathcal{C})$ the set of all walks ending in $v$ and by ${}_v\!\Walk_u(\mathcal{C})$ the intersection of ${}_v\!\Walk(\mathcal{C})$ and $\Walk_u(\mathcal{C})$. If $F:\mathcal{C}\to \mathcal{D}$ is a functor, then $F$ induces a map $\Walk(\mathcal{C})\to \Walk(\mathcal{D})$, which we also denote by $F$.

Recall that a {\it groupoid} is a small category in which any morphism has an inverse. The {\it fundamental groupoid} $\G(\Lambda)$ of a $k$-graph $\Lambda$ can be constructed as follows (cf. \cite[Section 19.1]{schubert}). Set 
\[R:=\{(\lambda\lambda^*,1_{r(\lambda)}),(\lambda^*\lambda,1_{s(\lambda)}),(\lambda\circ\mu,\lambda\mu)\mid \lambda,\mu\in\Lambda, ~s(\lambda)=r(\mu)\}.\]
We define an equivalence relation $\sim_R$ on $\Walk(\Lambda)$ as follows. Let $p,p'\in\Walk(\Lambda)$. Then $p\sim_R p'$ if and only if there is a finite sequence $p=q_0,q_1,\dots,q_{n-1},q_n=p'$ in $\Walk(\Lambda)$ such that $q_i$ is constructed
from $q_{i-1}$ (for $i = 1, 2, \dots ,n$) as follows: some subpath $a$ of $q_{i-1}$ is replaced by a walk $b$ which has the property that $(a,b)\in R$ or $(b,a)\in R$.
The objects of $\G(\Lambda)$ are the objects of $\Lambda$. The morphisms of $\G(\Lambda)$ are the $\sim_R$-equivalence classes of $\Walk(\Lambda)$ (note that equivalent walks have the same source and range). The composition of morphisms in $\G(\Lambda)$ is induced by the composition of walks in $\Walk(\Lambda)$. The assignment $\Lambda\mapsto \G(\Lambda)$ is functorial from $k$-graphs to groupoids.

There is a canonical functor $i:\Lambda\to\mathcal{G}(\Lambda)$ which is the identity on objects and maps a morphism $\lambda$ to $[\lambda]_{\sim_R}$. The functor $i$ has the following universal property: for any functor $T$ from $\Lambda$ to a groupoid $\mathcal{H}$ there exists a unique functor $T':\G(\Lambda)\to\mathcal{H}$ making the diagram
\[\xymatrix{\Lambda\ar[r]^{i}\ar[dr]_{T}&\G(\Lambda)\ar[d]^{T'}\\&\mathcal{H}}\]
commute.

\subsection{Coverings of $k$-graphs} \label{hnhfgftgrgr}

\begin{definition}\label{defcovering}
A {\it covering} of a $k$-graph $\Lambda$ is a pair $(\Omega,\alpha)$ consisting of a $k$-graph $\Omega$ and a $k$-graph morphism $\alpha:\Omega\rightarrow \Lambda$ such that (i) and (ii) below hold.
\begin{enumerate}[(i)]
\item For any $v\in \Omega^{\ob}$, $\alpha$ maps $\Omega v$ 1–1 onto $\Lambda \alpha(v)$.

\smallskip

\item For any $v\in \Omega^{\ob}$, $\alpha$ maps $v\Omega$ 1–1 onto $\alpha(v)\Lambda$.

\end{enumerate} If $(\Omega,\alpha)$ and $(\Sigma,\beta)$ are coverings of $\Lambda$, a {\it morphism} from $(\Omega,\alpha)$ to $(\Sigma,\beta)$ is a $k$-graph morphism
$\phi:\Omega\rightarrow\Sigma$ making the diagram
\[\xymatrix{\Omega\ar[rr]^{\phi}\ar[dr]_{\alpha}&&\Sigma\ar[dl]^{\beta}\\&\Lambda&}\]
commute. 
\end{definition}

\begin{definition}
Let $\Lambda$ be a $k$-graph. A covering $(\Omega,\alpha)$ of $\Lambda$ is {\it universal} if for any covering $(\Sigma,\beta)$ of $\Lambda$ there exists a morphism $(\Omega,\alpha)\to(\Sigma, \beta)$.
\end{definition}

\begin{theorem}[{\cite[Theorem 2.7]{PQR}}]
Every $k$-graph $\Lambda$ has a universal covering.
\end{theorem}

The {\it fundamental group} of $\Lambda$ at a vertex $x\in \Lambda^{\ob}$ is the group $\pi(\Lambda,x):= x\G(\Lambda)x$. By \cite[Theorems 2.2, 2.7, 2.8]{PQR} there is a 1-1 correspondence between the isomorphism classes of coverings of $\Lambda$ and the conjugacy classes of subgroups of $\pi(\Lambda,x)$. If $\alpha:\Omega\to\Lambda$ is a $k$-graph morphism and $v\in\Omega^{\ob}$, then there is a group homomorphism $\alpha_* :\pi(\Omega,v) \to \pi(\Lambda,\alpha(v))$ induced by $\alpha$. If $(\Omega,\alpha)$ is a covering of $\Lambda$, then $\alpha_* :\pi(\Omega,v) \to \pi(\Lambda,\alpha(v))$ is injective.

\section{Representation $k$-graphs} 
In this section $\Lambda$ denotes a fixed $k$-graph. 

\subsection{Representation $k$-graphs} Below we introduce the main notion of this paper, namely a representation $k$-graph for a given $k$-graph. 

\begin{definition}\label{defwp}
A {\it representation $k$-graph} for $\Lambda$ is a pair $(\Delta,\alpha)$ consisting of a $k$-graph $\Delta$ and a $k$-graph morphism $\alpha:\Delta\to\Lambda$ such that (i) and (ii) below hold.
\begin{enumerate}[(i)]
\item For any $v\in \Delta^{\ob}$, $\alpha$ maps $\Delta v$ 1–1 onto $\Lambda \alpha(v)$.

\smallskip

\item For any $v\in \Delta^{\ob}$ and $n\in \N^k$, $v\Delta^n$ is a singleton.

\end{enumerate}
If $(\Delta,\alpha)$ and $(\Sigma,\beta)$ are representation $k$-graphs for $\Lambda$, a {\it morphism} from $(\Delta,\alpha)$ to $(\Sigma,\beta)$ is a $k$-graph morphism
$\phi:\Delta\rightarrow\Sigma$ making the diagram
\[\xymatrix{\Delta\ar[rr]^{\phi}\ar[dr]_{\alpha}&&\Sigma\ar[dl]^{\beta}\\&\Lambda&}\]
commute. 
\end{definition}

We will see in Section 4 that any representation $k$-graph for $\Lambda$ yields a module for the Kumjian-Pask algebra $\KP(\Lambda)$. The irreducible representation $k$-graphs defined below are precisely those representation $k$-graphs that yield a simple module.

\begin{definition}\label{defrepirres}
Let $(\Delta,\alpha)$ be a representation $k$-graph for $\Lambda$. Then $(\Delta,\alpha)$ is called {\it irreducible} if $\alpha(\Walk_u(\Delta))\neq \alpha(\Walk_v(\Delta))$ for any $u\neq v\in \Delta^{\ob}$.
\end{definition} 

We denote by $\RG(\Lambda)$ the category of representation $k$-graphs for $\Lambda$.  The lemma below will be used quite often in the sequel.

\begin{lemma}\label{lemwell} Let $(\Delta,\alpha)$ be an object of $\RG(\Lambda)$. Let $p, q\in \Walk(\Delta)$ such that $\alpha(p)=\alpha(q)$. If $s(p)=s(q)$ or $r(p)=r(q)$, then $p=q$.
\end{lemma} 
\begin{proof}
Clearly $p=x_n\dots x_1$ and $q=y_n\dots y_1$, for some $n\geq 1$ and $x_1,\dots,x_n, y_1,\dots, y_n\in \Delta\cup\Delta^*$. First suppose that $s(p)=s(q)$. We proceed by induction on $n$.

Case $n=1$: Suppose $\alpha(x_1)=\alpha(y_1)=\lambda$ for some $\lambda\in\Lambda$. It follows from Definition \ref{defwp}(i) that $x_1=y_1$ and hence $p=q$. Suppose now that $\alpha(x_1)=\alpha(y_1)=\lambda^*$ for some $\lambda\in\Lambda$. Then it follows from Definition \ref{defwp}(ii) that $x_1=y_1$ and hence $p=q$.

Case $n\to n+1$:
Suppose that $p=x_{n+1}\dots  x_1$ and $q=y_{n+1}\dots y_1$. By the inductive assumption we have $x_i=y_i$ for any $1\leq i\leq n$. It follows that $r(x_n)=r(y_n)=:v$. Clearly $x_{n+1}, y_{n+1}\in \Walk_v(\Delta)$. Now we can apply the case $n=1$ and obtain $x_{n+1}=y_{n+1}$.

Now suppose that $r(p)=r(q)$. Then $s(p^*)=s(q^*)$. Since clearly $\alpha(p^*)=\alpha(q^*)$, we obtain $p^*=q^*$. Hence $p=q$.
\end{proof}

The next lemma is easy to check.

\begin{lemma}\label{hfgfghhffeee}
If $(\Omega,\alpha)$ is a covering of $\Lambda$ and $(\Delta,\beta)$ a representation $k$-graph for $\Omega$, then $(\Delta,\alpha\circ\beta)$ is a representation $k$-graph for $\Lambda$. On the other hand, if $(\Delta,\alpha)$ is a representation $k$-graph for $\Lambda$ and $(\Omega,\beta)$ a covering of $\Delta$, then $(\Omega,\alpha\circ\beta)$ is a representation $k$-graph for $\Lambda$.
\end{lemma}

\begin{proposition}\label{propsurj}
Let $\phi:(\Delta,\alpha)\to (\Sigma,\beta)$ be a morphism in $\RG(\Lambda)$. Then $(\Delta,\phi)$ is a covering of $\Sigma$. 
\end{proposition}
\begin{proof}
Let let $v\in\Delta^{\ob}$. Since $\Delta$ and $\Sigma$ satisfy condition (i) in Definition \ref{defwp}, the maps $\alpha|_{\Delta v}:\Delta v\to \Lambda \alpha(v)$ and $\beta|_{\Sigma \phi(v)}:\Sigma\phi(v)\to \Lambda \alpha(v)$ are bijective. It follows that 
\begin{align*}
&\beta\circ\phi=\alpha\\
\Rightarrow~&(\beta\circ\phi)|_{\Delta v}=\alpha|_{\Delta v}\\
\Rightarrow~&\beta|_{\Sigma\phi(v)}\circ\phi|_{\Delta v}=\alpha|_{\Delta v}\\
\Rightarrow~&\phi|_{\Delta v}=(\beta|_{\Sigma\phi(v)})^{-1}\circ\alpha|_{\Delta v}.
\end{align*}
Hence $\phi|_{\Delta v}:\Delta v\to \Sigma\phi(v)$ is bijective, i.e. $\phi$ maps $\Delta v$ 1–1 onto $\Sigma \phi(v)$.\\
It remains to show that $\phi$ maps $v\Delta$ 1–1 onto $\phi(v)\Sigma$. But this follows from the fact that $v\Delta=\bigsqcup_{n\in\N^k}v\Delta^n$, $\phi(v)\Sigma=\bigsqcup_{n\in\N^k}\phi(v)\Sigma^n$, each of the sets $v\Delta^n$ and $\phi(v)\Sigma^n$ is a singleton (by condition (ii) in Definition \ref{defwp}) and $\phi$ is a degree-preserving functor.
\end{proof}

\subsection{Quotients of representation $k$-graphs}
For any object $(\Delta,\alpha)$ of $\RG(\Lambda)$ we define an equivalence relation $\sim$ on $\Delta^{\ob}$ by $u\sim v$ if $\alpha(\Walk_u(\Delta))=\alpha(\Walk_v(\Delta))$. Recall that if $\sim$ and $\approx$ are equivalence relations on a set $X$, then one writes $\approx~\leq~ \sim $ (and calls $\approx$ {\it finer} than $\sim$, and $\sim$ {\it coarser} than $\approx$) if $x\approx y$ implies that $x\sim y$, for any $x,y\in X$. 

\begin{definition}\label{defadm}
Let $(\Delta,\alpha)$ be an object of $\RG(\Lambda)$. An equivalence relation $\approx$ on $\Delta^{\ob}$ is called {\it admissible} if (i) and (ii) below hold.
\begin{enumerate}[(i)]
\item $\approx~\leq~\sim$.
\item If $u\approx v$, $p\in{} _x\!\Walk_u(\Delta)$, $q\in{} _y\!\Walk_v(\Delta)$ and $\alpha(p)=\alpha(q)$, then $x\approx y$.
\end{enumerate}
\end{definition}

The lemma below is easy to check.
\begin{lemma}\label{lemlattice}
The admissible equivalence relations on $\Delta^{\ob}$ (with partial order $\leq$) form a bounded lattice whose maximal element is $\sim$ and whose minimal element is the equality relation $=$.
\end{lemma}

Let $(\Delta,\alpha)$ be an object of $\RG(\Lambda)$ and $\approx$ an admissible equivalence relation on $\Delta^{\ob}$. We define an equivalence relation $\approx$ on $\Delta$ by $\delta\approx \delta'$ if $s(\delta)\approx s(\delta')$ and $\alpha(\delta)=\alpha(\delta')$. Define a $k$-graph $(\Delta_\approx,\alpha_\approx)$ by 
\begin{align*}
\Delta_\approx^{\ob}&=\Delta^{\ob}/\approx,\\
\Delta_\approx&=\Delta/\approx,\\
s([\delta])&=[s(\delta)],\\
r([\delta])&=[r(\delta)],\\
d([\delta])&=d(\delta).
\end{align*}
The composition of morphisms in $\Delta_\approx$ is defined as follows. Let $[\delta],[\delta']\in \Delta/\approx$ such that $s([\delta])=r([\delta'])$. Then $s(\delta)\approx r(\delta')$ whence $s(\delta)\sim r(\delta')$, i.e. $\alpha(\Walk_{s(\delta)}(\Delta))=\alpha(\Walk_{r(\delta')}(\Delta))$. This implies that there is a $\delta''\in \Delta r(\delta')$ such that $\alpha(\delta'')=\alpha(\delta)$. Note that $[\delta'']=[\delta]$. We define $[\delta]\circ[\delta']:=[\delta''\circ\delta']$. One checks easily that this composition is well-defined. The identity morphisms are defined by $1_{[v]}=[1_v]$. Moreover, we define a $k$-graph morphism $\alpha_\approx:\Delta_\approx\to\Lambda$ by $\alpha_\approx([v])=\alpha(v)$ and $\alpha_\approx([\delta])=\alpha(\delta)$ for any $v\in \Delta^{\ob}$ and $\delta\in\Delta$. We leave it to the reader to check that $(\Delta_\approx,\alpha_\approx)$ is a representation $k$-graph for $\Lambda$. We call $(\Delta_\approx,\alpha_\approx)$ a {\it quotient} of $(\Delta,\alpha)$. 

\begin{lemma}\label{lemcompare}
Let $(\Delta,\alpha)$ be an object of $\RG(\Lambda)$. Let $\approx~\leq ~\approx'$ be admissible equivalence relations on $\Delta^{\ob}$. Then there is a morphism $(\Delta_\approx,\alpha_\approx)\to (\Delta_{\approx'},\alpha_{\approx'})$.
\end{lemma}
\begin{proof}
Define a $k$-graph morphism $\phi:\Delta_\approx\to \Delta_{\approx'}$ by $\phi([v]_\approx)=[v]_{\approx'}$ and $\phi([\delta]_\approx)=[\delta]_{\approx'}$ for any $v\in\Delta^{\ob}$ and $\delta\in\Delta$. Since $\approx~\leq ~\approx'$, $\phi$ is well-defined. Clearly $\alpha_{\approx'}\circ\phi=\alpha_{\approx}$ and therefore $\phi:(\Delta_\approx,\alpha_\approx)\to (\Delta_{\approx'},\alpha_{\approx'})$ is a morphism in $\RG(\Lambda)$. 
\end{proof}

\begin{lemma}\label{lempath}
Let $(\Delta,\alpha)$ and $(\Sigma,\beta)$ be objects of $\RG(\Lambda)$. Let $u\in \Delta^{\ob}$ and $v\in \Sigma^{\ob}$. If $\alpha(\Walk_u(\Delta))\subseteq \beta(\Walk_v(\Sigma))$, then $\alpha(\Walk_u(\Delta))=\beta(\Walk_v(\Sigma))$.
\end{lemma}
\begin{proof}
Suppose that $\alpha(\Walk_u(\Delta))\subseteq \beta(\Walk_v(\Sigma))$. It follows that $\alpha(u)=\beta(v)$. We have to show that $\beta(\Walk_v(\Sigma))\subseteq \alpha(\Walk_u(\Delta))$. Let $p\in\Walk_v(\Sigma)$. Then $p=y_n\dots y_1$ for some $y_1,\dots, y_n\in \Sigma\cup\Sigma^*$ where $n\geq 1$. We proceed by induction on $n$.

Case $n=1$: Suppose that $p=\sigma$ for some $\sigma\in \Sigma v$. Then $\beta(\sigma)=\lambda$ for some $\lambda\in \Lambda\beta(v)$. Since $(\Delta,\alpha)$ satisfies condition (i) in Definition \ref{defwp}, there is a (unique) $\delta\in \Delta u$ such that $\alpha(\delta)=\lambda$. Hence $\beta(p)=\beta(\sigma)=\lambda=\alpha(\delta)\in \alpha(\Walk_u(\Delta))$.

Suppose now that $p=\sigma^*$ for some $\sigma\in v\Sigma$. Set $m:=d(\sigma)$. Since $(\Delta,\alpha)$ satisfies condition (ii) in Definition \ref{defwp}, there is a $\delta\in u\Delta^m$. Since $\alpha(\Walk_u(\Delta))\subseteq \beta(\Walk_v(\Sigma))$, there is a $\sigma'\in v\Sigma^m$ such that $\alpha(\delta)=\beta(\sigma')$. Clearly $\sigma'=\sigma$ since $(\Sigma,\beta)$ satisfies condition (ii) in Definition \ref{defwp}. Hence $\beta(p)=\beta(\sigma^*)=\alpha(\delta^*)\in \alpha(\Walk_u(\Delta))$.

Case $n\rightarrow n+1$:  Suppose $p=y_{n+1}y_n\dots y_1$. By the induction assumption we know that $\beta(y_n\dots y_1)\in \alpha(\Walk_u(\Delta))$. Hence $\beta(y_n\dots y_1)=\alpha(x_n\dots x_1)$ for some walk $x_n\dots x_1\in \Walk_u(\Delta)$. Set $u':=r(x_n)$ and $v':=r(y_n)$. Clearly $\alpha(\Walk_{u'}(\Delta))\subseteq \beta(\Walk_{v'}(\Sigma))$. Applying the case $n=1$ we obtain that $\beta(y_{n+1})\in \alpha(\Walk_{u'}(\Delta))$. Hence $\beta(y_{n+1})=\alpha(x_{n+1})$ for some $x_{n+1}\in \Walk_{u'}(\Delta)$. Thus $\beta(p)=\beta( y_{n+1}y_n\dots y_1)=\alpha(x_{n+1}x_n\dots x_1)\in \alpha(\Walk_u(\Delta))$.
\end{proof} 

\begin{proposition}\label{thmadm}
Let $(\Delta,\alpha)$ and $(\Sigma,\beta)$ be objects of $\RG(\Lambda)$. Then there is a morphism $\phi:(\Delta,\alpha)\to (\Sigma,\beta)$ if and only if $(\Sigma,\beta)$ is isomorphic to a quotient of $(\Delta,\alpha)$.
\end{proposition}
\begin{proof}
$(\Rightarrow)$ Suppose there is a morphism $\phi:(\Delta,\alpha)\to (\Sigma,\beta)$. If $u,v\in \Delta^{\ob}$, we write $u\approx v$ if $\phi(u)=\phi(v)$. Clearly $\approx$ defines an equivalence relation on $\Delta^{\ob}$. Below we check that $\approx$ is admissible. 
\begin{enumerate}[(i)]
\item Suppose $u\approx v$. Then $\alpha(\Walk_u(\Delta))=\beta(\Walk_{\phi(u)}(\Sigma))=\beta(\Walk_{\phi(v)}(\Sigma))=\alpha(\Walk_v(\Delta))$ by Lemma \ref{lempath}. Hence $u\sim v$.
\item Suppose $u\approx v$, $p\in{} _x\!\Walk_u(\Delta)$, $q\in{} _y\!\Walk_v(\Delta)$ and $\alpha(p)=\alpha(q)$. Clearly $\phi(p)\in{} _{\phi(x)}\!\Walk_{\phi(u)}(\Sigma)$ and $\phi(q)\in{} _{\phi(y)}\!\Walk_{\phi(v)}(\Sigma)$. Moreover, $\beta(\phi(p))=\alpha(p)=\alpha(q)=\beta(\phi(q))$. Since $\phi(u)=\phi(v)$, it follows from Lemma \ref{lemwell} that $\phi(p)=\phi(q)$. Hence $\phi(x)=r(\phi(p))=r(\phi(q))=\phi(y)$ and therefore $x\approx y$.
\end{enumerate}

Note that by Lemma \ref{lemwell} we have $\delta\approx \delta'$ if and only if $\phi(\delta)=\phi(\delta')$, for any $\delta,\delta'\in \Delta$. Define a $k$-graph morphism $\psi:\Delta_\approx\to \Sigma$ by $\psi([v])=\phi(v)$ and $\psi([\delta])=\phi(\delta)$ for any $v\in\Delta^{\ob}$ and $\delta\in\Delta$. Clearly $\beta\circ\psi=\alpha_{\approx}$ and therefore $\psi:(\Delta_\approx,\alpha_\approx)\to (\Sigma,\beta)$ is a morphism in $\RG(\Lambda)$. In view of Proposition \ref{propsurj}, $\psi$ is bijective and hence $\psi$ is an isomorphism.

$(\Leftarrow)$ Suppose now that $(\Sigma,\beta)\cong (\Delta_\approx,\alpha_\approx)$ for some admissible equivalence relation $\approx$ on $\Delta^{\ob}$. In order to show that there is a morphism $\alpha:(\Delta,\alpha)\to (\Sigma,\beta)$ it suffices to show that there is a morphism $\beta:(\Delta,\alpha)\to (\Delta_\approx,\alpha_\approx)$. But this is obvious (define $\beta(v)=[v]$ and $\beta(\delta)=[\delta]$).
\end{proof}

\subsection{The connected components of the category $\RG(\Lambda)$} \label{subsecnn}
Recall that any category $\mathcal{C}$ can be written as a disjoint union (or coproduct) of a collection of connected categories, which are called the {\it connected components} of $\mathcal{C}$. Each connected component is a full subcategory of $\mathcal{C}$.

\begin{lemma}\label{lemrepell}
Let $(\Delta,\alpha)$ and $(\Sigma,\beta)$ be objects of $\RG(\Lambda)$ and suppose there is a morphism $(\Delta,\alpha)\to(\Sigma,\beta)$ or a morphism $(\Sigma,\beta)\to(\Delta,\alpha)$. Let $(\Omega,\tau)$ be a universal covering of $\Delta$. Then there is a $k$-graph morphism $\eta:\Omega\to\Sigma$ such that $(\Omega,\eta)$ is a universal covering of $\Sigma$ and $\eta:(\Omega,\alpha\circ \tau)\to(\Sigma,\beta)$ is a morphism in $\RG(\Lambda)$.
\end{lemma}
\begin{proof}
First suppose that there is a morphism $\phi:(\Delta,\alpha)\rightarrow(\Sigma,\beta)$. Since the diagram
\[\xymatrix@C=15pt@R=20pt{&\Omega\ar[dl]_{\tau}\ar[dr]^{\phi\circ \tau}&\\\Delta\ar[dr]_{\alpha}\ar[rr]^{\phi}&&\Sigma\ar[dl]^{\beta}\\&\Lambda&}\]
commutes, $\phi\circ \tau:(\Omega,\alpha\circ \tau)\to(\Sigma,\beta)$ is a morphism in $\RG(\Lambda)$. It follows from Proposition \ref{propsurj}, that $(\Omega,\phi\circ \tau)$ is a covering of $\Sigma$. By \cite[Theorem 2.7]{PQR} there is an $x\in \Delta^{\ob}$ and a $v\in \tau^{-1}(x)$ such that $\tau_*\pi(\Omega,v)=\{x\}$. Hence 
\[(\phi\circ \tau)_*\pi(\Omega,v)=\phi_*(\tau_*\pi(\Omega,v))=\phi_*(\{x\})=\{\phi(x)\}.\]
It follows that $(\Omega,\phi\circ \tau)$ is a universal covering of $\Sigma$, again by \cite[Theorem 2.7]{PQR}. 
\\
Suppose now that there is a morphism $\phi:(\Sigma,\beta)\rightarrow (\Delta,\alpha)$. Let $(\Omega',\tau')$ be a universal covering of $\Sigma$. Then $(\Omega',\phi\circ \tau')$ is a universal covering of $\Delta$ by the previous paragraph. It follows from \cite[Theorems 2.2, 2.7]{PQR} that $(\Omega, \tau)\cong (\Omega', \phi\circ \tau')$, i.e. there is a $k$-graph isomorphism $\gamma:\Omega\to \Omega'$ making the diagram
\[\xymatrix@C=15pt@R=20pt{\Omega\ar[d]_{\tau}\ar[rr]^{\gamma}&&\Omega'\ar[d]^
{\tau'}\\\Delta\ar[dr]_{\alpha}&&\Sigma\ar[dl]^{\beta}\ar[ll]_{\phi}\\&\Lambda&}\]
commute. It follows that $\tau'\circ\gamma:(\Omega,\alpha\circ \tau)\to(\Sigma,\beta)$ is a morphism in $\RG(\Lambda)$. One checks easily that $(\Omega,\tau'\circ\gamma)$ is a universal covering of $\Sigma$. 
\end{proof}

Let $C$ be a connected component of $\RG(\Lambda)$. Choose an object $(\Delta,\alpha)$ of $C$ and a universal covering $(\Omega,\tau)$ of $\Delta$. By Lemma \ref{hfgfghhffeee}, $(\Omega,\alpha\circ\tau)$ is an object of $C$. We set 
\[(\Omega_C,\zeta_C):=(\Omega,\alpha\circ\tau)~\text{ and }~(\Gamma_C,\xi_C):=((\Omega_C)_\sim,(\zeta_C)_\sim).\]
We call an object $X$ in a category $\mathcal{C}$ {\it repelling} (resp. {\it attracting}) if for any object $Y$ in $\mathcal{C}$ there is a morphism $X\rightarrow Y$ (resp. $Y\rightarrow X$). 

\begin{theorem}\label{thmm1}
Let $C$ be a connected component of $\RG(\Lambda)$. Then $(\Omega_C,\zeta_C)$ is a repelling object of $C$, and consequently the objects of $C$ are up to isomorphism precisely the quotients of $(\Omega_C,\zeta_C)$.
\end{theorem}
\begin{proof}
Let $(\Sigma,\gamma)$ be an object of $C$. Then there is a sequence of objects \[(\Delta,\alpha)=(\Delta_0,\alpha_0),(\Delta_1,\alpha_1),\dots,(\Delta_{n-1},\alpha_{n-1}),(\Delta_n,\alpha_n)=(\Sigma,\gamma)\]
of $C$ such that for each $0\leq i\leq n-1$ there is a morphism $(\Delta_i,\alpha_i)\to(\Delta_{i+1},\alpha_{i+1})$ or a morphism $(\Delta_{i+1},\alpha_{i+1})\to(\Delta_{i},\alpha_{i})$. Set $\eta_0:=\tau$. By inductively applying Lemma \ref{lemrepell} we obtain $k$-graph morphisms $\eta_1:\Omega\to \Delta_1,~\eta_2:\Omega\to\Delta_2,~\dots~,\eta_n:\Omega\to\Delta_n$ such that for any $1\leq i\leq n$, $(\Omega,\eta_i)$ is a universal covering of $\Delta_i$ and $\eta_i:(\Omega,\alpha_{i-1}\circ \eta_{i-1})\to (\Delta_{i},\alpha_{i})$ is a morphism in $\RG(\Lambda)$. Since the diagram
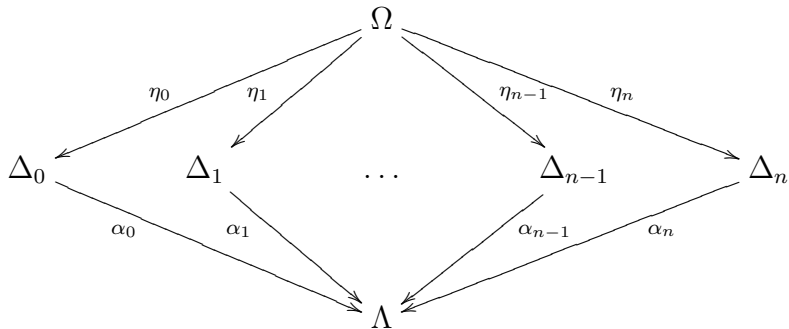
\[\xymatrix@C=45pt@R=40pt{&&\Omega\ar[dll]_(0.6){\eta_0}\ar[dl]_(0.6){\eta_1}\ar[dr]^(0.6){\eta_{n-1}}\ar[drr]^(0.6){\eta_n}&&\\\Delta_0\ar[drr]_(0.3){\alpha_0}&\Delta_1\ar[dr]_(0.3){\alpha_1}&\dots&\Delta_{n-1}\ar[dl]^(0.3){\alpha_{n-1}}&\Delta_n\ar[dll]^(0.3){\alpha_n}\\&&\Lambda&&}\]
is commutative, we obtain that $\eta_n:(\Omega_C,\zeta_C)=(\Omega,\alpha_0\circ\eta_0)\to (\Delta_n,\alpha_n)=(\Sigma,\gamma)$ is a morphism in $\RG(\Lambda)$. Thus $(\Omega_C,\zeta_C)$ is a repelling object of $C$. The second statement now follows from Proposition \ref{thmadm}.
\end{proof}

\begin{theorem}\label{thmm2}
Let $C$ be a connected component of $\RG(\Lambda)$. Then $(\Gamma_C,\xi_C)$ is an attracting object of $C$, and consequently the objects of $C$ are precisely the representation $k$-graphs $(\Sigma,\xi_C\circ \eta)$ where $(\Sigma,\eta)$ is a covering of $\Gamma_C$. 
\end{theorem}
\begin{proof}
The first statement of the theorem follows from Lemma \ref{lemcompare} and Theorem \ref{thmm1}. The second statement now follows from Lemma \ref{hfgfghhffeee} and Proposition \ref{propsurj}. 
\end{proof}

\begin{corollary}\label{corm2}
Let $C$ be a connected component of $\RG(\Lambda)$. Then up to isomorphism $(\Gamma_C,\xi_C)$ is the unique irreducible representation $k$-graph in $C$. 
\end{corollary}
\begin{proof}
We leave it to the reader to check that $(\Gamma_C,\xi_C)$ is irreducible. Let now $(\Sigma,\gamma)$ be an irreducible representation $k$-graph in $C$. It follows from Proposition \ref{thmadm} and Theorem \ref{thmm2} that $(\Gamma_C,\xi_C)$ is isomorphic to a quotient of $(\Sigma,\gamma)$. But since $(\Sigma,\gamma)$ is irreducible, there is only one admissible equivalence relation on $\Sigma^{\ob}$, namely the equality relation $=$, and the corresponding quotient $(\Sigma_{=},\gamma_{=})$ is isomorphic to $(\Sigma,\gamma)$.
\end{proof}

\section{Modules for Kumjian-Pask algebras via representation $k$-graphs}
In this section $\Lambda$ denotes a fixed $k$-graph. 
\subsection{Kumjian-Pask algebras}
For each $\lambda\in\Lambda$ of degree $\neq 0$ we introduce a symbol $\lambda^*$. For each $\lambda\in\Lambda^0$ we set $\lambda^*:=\lambda$.

\begin{definition}\label{defKP}
The $K$-algebra $\KP(\Lambda)$ presented by the generating set $\Lambda\cup \Lambda^*$ and the relations 
\begin{enumerate}[(KP1)]
\item $\lambda\mu = \delta_{s(\lambda),r(\mu)}(\lambda\circ\mu)$ for any $\lambda,\mu\in\Lambda$,
\smallskip 
\item $\mu^*\lambda^*=\delta_{s(\lambda),r(\mu)}(\lambda\circ\mu)^*$ for any $\lambda,\mu\in\Lambda$,
\smallskip 
\item $\lambda^*\mu=\delta_{\lambda,\mu}1_{s(\lambda)}$ for any $\lambda,\mu\in\Lambda$ with $d(\lambda)=d(\mu)$, 
\smallskip 
\item  $\sum\limits_{\lambda\in v\Lambda^n}\lambda\lambda^*=1_v$ for any $v\in\Lambda^{\ob}$ and $n\in\mathbb{N}^k$
\end{enumerate}
is called the {\it Kumjian-Pask algebra} of $\Lambda$.
\end{definition}

We may view the walks in $\Lambda$ as monomials in $\KP(\Lambda)$. Clearly any element of $\KP(\Lambda)$ is a $K$-linear combination of walks. 

\begin{remark}
The algebra $\KP(\Lambda)$ defined in Definition \ref{defKP} above is isomorphic to the algebra $\KP_K(\Lambda)$ defined in \cite[Definition 6.1]{ahls}. Note that the relations 
\[r(\lambda)\lambda=\lambda=\lambda s(\lambda)\text{ and }s(\lambda)\lambda^*=\lambda^*=\lambda^* r(\lambda)\text{ for any }\lambda\in\Lambda\]
in \cite[Definition 6.1]{ahls} are redundant. 
\end{remark}


\subsection{The functor $V$}
For an object $(\Delta,\alpha)$ of $\RG(\Lambda)$, let $V_{(\Delta,\alpha)}$ be the $K$-vector space with basis $\Delta^{\ob}$. 
For any $\lambda\in \Lambda$ define two endomorphisms $\sigma_{\lambda}$ and $\sigma_{\lambda^*}\in \End_K(V_{(\Delta,\alpha)})$ by
\begin{align*}
\sigma_{\lambda}(v)     &=\begin{cases}r(\delta),\quad                       &\text{if }\exists \delta\in \Delta v\text{ such that }\alpha(\delta)=\lambda\\0,     & \text{otherwise} \end{cases},\\
\sigma_{\lambda^*}(v)  &=\begin{cases}s(\delta),\quad                        &\text{if }\exists \delta\in v\Delta\text{ such that }\alpha(\delta)=\lambda\\0,    & \text{otherwise} \end{cases},
\end{align*}
where $v\in \Delta^{\ob}$. Note that $\sigma_{\lambda}$ and $\sigma_{\lambda^*}$ are well-defined since for any $v\in \Delta^{\ob}$ the maps $\alpha|_{\Delta v}$ and $\alpha|_{v\Delta}$ are injective. One checks routinely that there is an algebra homomorphism 
$\pi:\KP(\Lambda)\longrightarrow \End_K(V_{(\Delta,\alpha)})$ such that $\pi(\lambda)=\sigma_{\lambda}$ and $\pi(\lambda^*)=\sigma_{\lambda^*}$ for any $\lambda\in\Lambda$. 
Clearly $V_{(\Delta,\alpha)}$ becomes a left $\KP(\Lambda)$-module by defining $a.x:= \pi(a)(x)$ for any $a\in \KP(\Lambda)$ and $x\in V_{(\Delta,\alpha)}$. We call $V_{(\Delta,\alpha)}$ the {\it $\KP(\Lambda)$-module defined by $(\Delta,\alpha)$}. A morphism $\phi:(\Delta,\alpha)\rightarrow(\Sigma,\beta)$ in $\RG(\Lambda)$ induces a surjective $\KP(\Lambda)$-module homomorphism $V_\phi:V_{(\Delta,\alpha)}\rightarrow V_{(\Sigma,\beta)}$ such that $V_\phi(u)=\phi(u)$ for any $u\in \Delta^{\ob}$. This gives rise to a functor 
\[V:\RG(\Lambda)\to \Modd(\KP(\Lambda))\]
where $\Modd(\KP(\Lambda))$ denotes the category of left $\KP(\Lambda)$-modules.

The following lemma describes the action of $\Walk(\Lambda)$ on $V_{(\Delta,\alpha)}$. Note that by Lemma \ref{lemwell}, for any $p\in \Walk(\Lambda)$ and $u\in\Delta^{\ob}$ there is at most one $v\in\Delta^{\ob}$ such that $p\in\alpha({}_v\!\Walk_u(\Delta))$.
\begin{lemma}\label{lemaction}
Let $(\Delta,\alpha)$ be an object $(\Delta,\alpha)$ of $\RG(\Lambda)$. If $p\in \Walk(\Lambda)$ and $u\in\Delta^{\ob}$, then
\[p.u=\begin{cases}v,\quad\quad            &\text{if }p\in\alpha({}_v\!\Walk_u(\Delta)) \text{ for some }v\in \Delta^{\ob}, \\
0,                                           & \text{otherwise.} \end{cases}\]
\end{lemma}

\begin{corollary}\label{coraction}
If $a=\sum_{p\in \Walk(\Lambda)}k_p p\in \KP(\Lambda)$ and $u\in\Delta^{\ob}$, then
\[a.u=\sum_{v\in\Delta^{\ob}}(\sum_{p\in\alpha({}_v\!\Walk_u(\Delta))}k_p)v.\]
\end{corollary}

The corollory below follows from Lemmas \ref{lemwell} and \ref{lemaction}.
\begin{corollary}\label{corwell}
If $p\in \Walk(\Lambda)$ and $u\neq u'\in\Delta^{\ob}$, then either $p.u=p.u'=0$ or $p.u\neq p.u'$. 
\end{corollary}

\subsection{Fullness of the functor $V$}

Let $(\Delta,\alpha)$ be an object of $\RG(\Lambda)$ and $(\Delta_\approx,\alpha_\approx)$ a quotient of $(\Delta,\alpha)$. Then, by Proposition \ref{thmadm}, there is a morphism $\phi:(\Delta,\alpha)\to (\Delta_\approx,\alpha_\approx)$, and hence a surjective morphism $V_\phi:V_{(\Delta,\alpha)}\rightarrow V_{(\Delta_\approx,\alpha_\approx)}$. By the lemma below, which is easy to check, there is also a morphism $V_{(\Delta_\approx,\alpha_\approx)}\to V_{(\Delta,\alpha)}$.

\begin{lemma}\label{lemopp}
Let $(\Delta,\alpha)$ be an object of $\RG(\Lambda)$ and $(\Delta_\approx,\alpha_\approx)$ a quotient of $(\Delta,\alpha)$. Then there is a morphism $V_{(\Delta_\approx,\alpha_\approx)}\to V_{(\Delta,\alpha)}$ mapping $[u]\mapsto \sum_{v\approx u}v$.
\end{lemma}

The example below shows that in general $V$ is not full, namely, there can be morphisms $V_{(\Delta,\alpha)}\rightarrow V_{(\Sigma,\beta)}$ that are not induced by a morphism $(\Delta,\alpha)\rightarrow(\Sigma,\beta)$. 

\begin{example}\label{excat11}
Suppose $\Lambda$ is the $1$-graph with one object $v$ and one morphism of degree $1$, namely $\lambda$. Let $\Delta$ be the $1$-graph with two objects $v_1$ and $v_2$ whose only morphisms of degree $1$ are $\delta_1$, with source $v_1$ and range $v_2$, and $\delta_2$, with source $v_2$ and range $v_1$. Let $\alpha:\Delta\to\Lambda$ be the unique $1$-graph morphism. Then $(\Delta,\alpha)$ is a representation $k$-graph for $\Lambda$. By Lemma \ref{lemopp} there is a a homomorphism $V_{\Delta_\sim}\to V_\Delta$. Clearly this homomorphism is not induced by a morphism $(\Delta_\sim,\phi_\sim)\to (\Delta,\phi)$ (otherwise $(\Delta,\phi)$ would be isomorphic to a quotient of $(\Delta_\sim,\phi_\sim)$; but this is impossible since $\Delta$ has two objects while $\Delta_\sim\cong\Lambda$ has only one).
\end{example}

\begin{question}\label{Q1}
Can it happen that $(\Delta,\alpha)\not\cong(\Sigma,\beta)$ in $\RG(\Lambda)$ but $V_{(\Delta,\alpha)}\cong V_{(\Sigma,\beta)}$ in $\Modd(\KP(\Lambda))$?
\end{question}

The author does not know the answer to Question \ref{Q1}. But we will show that if $V_{(\Delta,\alpha)}\cong V_{(\Sigma,\beta)}$, then $(\Delta,\alpha)$ and $(\Sigma,\beta)$ lie in the same connected component of $\RG(\Lambda)$.

If $(\Delta,\alpha)$ and $(\Sigma,\beta)$ are objects of $\RG(\Lambda)$, we write $(\Delta,\alpha)\leftrightharpoons(\Sigma,\beta)$ if there is a $u\in\Delta^{\ob}$ and a $v\in\Sigma^{\ob}$ such that $\alpha(\Walk_u(\Delta))=\beta(\Walk_{v}(\Sigma))$. One checks easily that $\leftrightharpoons$ defines an equivalence relation on $\Ob(\RG(\Lambda))$. We leave the proof of the next lemma to the reader.

\begin{lemma}\label{lemwithoutnumber}
Let $(\Delta,\alpha)$ and $(\Sigma,\beta)$ be objects of $\RG(\Lambda)$. Let $u\in\Delta^{\ob}$ and $v\in\Sigma^{\ob}$ such that $\alpha(\Walk_u(\Delta))=\beta(\Walk_{v}(\Sigma))$. Then $\alpha(\Walk_{p.u}(\Delta))=\beta(\Walk_{p.v}(\Sigma))$ for any $p\in \alpha(\Walk_u(\Delta))=\beta(\Walk_{v}(\Sigma))$.
\end{lemma}

\begin{proposition}\label{prop41.5}
Let $(\Delta,\alpha)$ and $(\Sigma,\beta)$ be objects of $\RG(\Lambda)$. Then $(\Delta,\alpha)$ and $(\Sigma,\beta)$ lie in the same connected component of $\RG(\Lambda)$ if and only if $(\Delta,\alpha)\leftrightharpoons(\Sigma,\beta)$.
\end{proposition}
\begin{proof}
($\Rightarrow$) Suppose that $(\Delta,\alpha)$ and $(\Sigma,\beta)$ lie in the same connected component of $\RG(\Lambda)$. In order to prove that $(\Delta,\alpha)\leftrightharpoons(\Sigma,\beta)$, it suffices to consider the case that there is a morphism $\phi:(\Delta,\alpha)\to(\Sigma,\beta)$. Choose a $u\in\Delta^{\ob}$. Then clearly $\alpha(\Walk_u(\Delta))=\beta(\phi(\Walk_u(\Delta))\subseteq\beta(\Walk_{\phi(u)}(\Sigma))$. It follows from Lemma \ref{lempath} that $\alpha(\Walk_u(\Delta))=\beta(\Walk_{\phi(u)}(\Sigma))$. Thus $(\Delta,\alpha)\leftrightharpoons(\Sigma,\beta)$.\\
($\Leftarrow$) Suppose now that $(\Delta,\alpha)\leftrightharpoons(\Sigma,\beta)$. Then there is a $u_0\in\Delta^{\ob}$ and a $v_0\in\Sigma^{\ob}$ such that $\alpha(\Walk_{u_0}(\Delta))=\beta(\Walk_{v_0}(\Sigma))$. Since $\Delta$ is connected, we can choose for any $u\in \Delta^{\ob}$ a $p_u\in {}_u\!\Walk_{u_0}(\Delta)$. Define a functor $\phi:\Delta_\sim\to\Sigma_\sim$ by 
\begin{align*}
&\phi([u])=[\alpha(p_u).v_0] \text{ for any }u\in \Delta^{\ob}\text{ and }\\&\phi([\delta])=[\sigma]\text{ for any }\delta\in \Delta^{\ob},\text{ where }s(\sigma)=\alpha(p_{s(\delta)}).v_0\text{ and }\beta(\sigma)=\alpha(\delta).
\end{align*}
It follows from Lemma \ref{lemwithoutnumber} that $\alpha$ is well-defined. One checks routinely that $\phi:(\Delta_\sim,\alpha_\sim)\to(\Sigma_\sim,\beta_\sim)$ is a morphism in $\RG(\Lambda)$. Thus, in view of Proposition \ref{thmadm}, $(\Delta,\alpha)$ and $(\Sigma,\beta)$ lie in the same connected component of $\RG(\Lambda)$.
\end{proof}

\begin{lemma}\label{lemmodule}
Let $(\Delta,\alpha)$ and $(\Sigma,\beta)$ be objects of $\RG(\Lambda)$ and let $\theta:V_{(\Delta,\alpha)}\rightarrow V_{(\Sigma,\beta)}$ a $\KP(\Lambda)$-module homomorphism. Let $u\in \Delta^{\ob}$ and suppose that $\theta(u)=\sum_{i=1}^nk_iv_i$ for some $n\geq 1$, $k_1,\dots,k_n\in K^\times$ and pairwise distinct vertices $v_1,\dots,v_n\in\Sigma^{\ob}$. Then $\alpha(\Walk_u(\Delta))=\beta(\Walk_{v_i}(\Sigma))$ for any $1\leq i\leq n$.
\end{lemma}
\begin{proof} 
Let $p\in \Walk(\Lambda)$ such that $p\not\in\alpha(\Walk_u(\Delta))$. Then
\[0=\theta(0)=\theta(p.u)=p.\theta(u)=p.\sum_{i=1}^nk_iv_i=\sum_{i=1}^nk_i(p.v_i)\]
by Lemma \ref{lemaction}. It follows from Corollary \ref{corwell} that $p.v_i=0$ for any $1\leq i\leq n$, whence $p\not\in \beta(\Walk_{v_i}(\Sigma))$ for any $1\leq i\leq n$. Hence we have shown that $\alpha(\Walk_u(\Delta))\supseteq\beta(\Walk_{v_i}(\Sigma))$ for any $1\leq i\leq n$. It follows from Lemma \ref{lempath} that $\alpha(\Walk_u(\Delta))=\beta(\Walk_{v_i}(\Sigma))$ for any $1\leq i\leq n$.
\end{proof}

The theorem below follows directly from Proposition \ref{prop41.5} and Lemma \ref{lemmodule}.
\begin{theorem}\label{prop42}
Let $(\Delta,\alpha)$ and $(\Sigma,\beta)$ be objects of $\RG(\Lambda)$. If there is a nonzero $\KP(\Lambda)$-module homomorphism $\theta:V_{(\Delta,\alpha)}\rightarrow V_{(\Sigma,\beta)}$, then $(\Delta,\alpha)$ and $(\Sigma,\beta)$ lie in the same connected component of $\RG(\Lambda)$.
\end{theorem}

\begin{corollary}\label{prop43}
Let $(\Delta,\alpha)$ and $(\Sigma,\beta)$ be irreducible representation $k$-graphs for $\Lambda$. Then $(\Delta,\alpha)\cong(\Sigma,\beta)$ if and only if $V_{(\Delta,\alpha)}\cong V_{(\Sigma,\beta)}$.
\end{corollary}
\begin{proof}
Clearly $(\Delta,\alpha)\cong(\Sigma,\beta)$ implies $V_{(\Delta,\alpha)}\cong V_{(\Sigma,\beta)}$ since $V$ is a functor. Suppose now that $V_{(\Delta,\alpha)}\cong V_{(\Sigma,\beta)}$. Then, by Theorem \ref{prop42}, $(\Delta,\alpha)$ and $(\Sigma,\beta)$ lie in the same connected component $C$ of $\RG(\Lambda)$. It follows from Corollary \ref{corm2} that $(\Delta,\alpha)\cong (\Gamma_C,\xi_C)\cong (\Sigma,\beta)$.
\end{proof}

\subsection{Simplicity of the modules $V_{(\Delta,\alpha)}$}
In this subsection we show that the $\KP(\Lambda)$-module $V_{(\Delta,\alpha)}$ is simple if and only if $(\Delta,\alpha)$ is irreducible.

\begin{lemma}[{\cite[Lemma 63]{HPS}}]\label{lembasis}
Let $W$ be a $K$-vector space and $B$ a linearly independent subset of  $W$. Let $k_i\in K$ and $u_i,v_i\in B$, where $1\leq i \leq n$. Then $\sum_{i=1}^nk_i(u_i-v_i)\not\in B$.
\end{lemma}


\begin{theorem}\label{thmirr} Let $(\Delta,\alpha)$ be an object of $\RG(\Lambda)$
Then the following are equivalent.
\begin{enumerate}[\upshape(i)]
\item $V_{(\Delta,\alpha)}$ is simple.
\smallskip
\item For any $x\in V_{(\Delta,\alpha)}\setminus\{0\}$ there is an $a\in \KP(\Lambda)$ such that $a.x \in \Delta^{\ob}$.
\smallskip
\item For any $x\in V_{(\Delta,\alpha)}\setminus\{0\}$ there is a $k\in K$ and a $p\in \Walk(\Lambda), $ such that $kp.x\in \Delta^{\ob}$.
\smallskip
\item $(\Delta,\alpha)$ is irreducible.
\end{enumerate}
\end{theorem}
\begin{proof} 
(i) $\Longrightarrow$ (iv). Assume that there are $u\neq v\in \Delta^{\ob}$ such that $\alpha(\Walk_u(\Delta))= \alpha(\Walk_v(\Delta))$. Consider the submodule $\KP(\Lambda).(u-v)\subseteq V_{(\Delta,\alpha)}$. Since $V_{(\Delta,\alpha)}$ is simple by assumption, we have $\KP(\Lambda).(u-v)=V_{(\Delta,\alpha)}$. Hence there is an $a\in \KP(\Lambda)$ such that $a.(u-v)=v$. Clearly there is an $n\geq 1$, $k_1,\dots,k_n\in K^\times$ and pairwise distinct $p_1,\dots,p_n\in \Walk(\Lambda)$ such that $a=\sum_{i=1}^nk_ip_i$. We may assume that $p_i.(u-v)\neq 0$, for any $1\leq i\leq n$. It follows that $p_i\in \alpha(\Walk_u(\Delta))= \alpha(\Walk_v(\Delta))$, for any $i$ and moreover, that $p_i.(u-v)=u_i-v_i$ for some distinct $u_i,v_i\in \Delta^{\ob}$. Hence
\[v=a.(u-v)=(\sum_{i=1}^nk_ip_i).(u-v)=\sum_{i=1}^nk_i(u_i-v_i)\]
which contradicts Lemma \ref{lembasis}.

\medskip 
(iv) $\Longrightarrow$ (iii). Let $x\in V_{(\Delta,\alpha)}\setminus \{0\}$. Then there is an $n\geq 1$, $k_1,\dots,k_n\in K^\times$ and pairwise disjoint $v_1,\dots, v_n\in \Delta^{\ob}$ such that $x=\sum_{i=1}^nk_iv_i$. If $n=1$, then $k_1^{-1}\alpha(v_1).x = v_1$. Suppose now that $n>1$. By assumption, we can choose a $p_1\in \alpha(\Walk_{v_1}(\Delta))$ such that $p_1\not\in\alpha(\Walk_{v_2}(\Delta))$. Clearly $p_1.x\neq 0$ is a linear combination of at most $n-1$ vertices from $\Delta^{\ob}$. Proceeding this way, we obtain walks $p_1,\dots,p_m$ such that $p_m\dots p_1.x=kv$ for some $k\in K^\times$ and $v\in \Delta^{\ob}$. Hence $k^{-1}p_m\dots p_1.x=v$.

\medskip 

(iii) $\Longrightarrow$ (ii). This implication is trivial.

\medskip 

(ii) $\Longrightarrow$ (i). Let $U\subseteq V_{(\Delta,\alpha)}$ be a nonzero $\KP(\Lambda)$-submodule and $x\in U\setminus\{0\}$. By assumption, there is an $a\in \KP(\Lambda)$ and a $v\in \Delta^{\ob}$ such that $v=a.x\in U$. Let now $v'$ be an arbitrary vertex in $\Delta^{\ob}$. Since $\Delta$ is connected, there is a $p\in {}_{v'}\!\Walk_{v}(\Delta)$. It follows that $v'=\alpha(p).v\in U$. Hence $U$ contains $\Delta^{\ob}$ and thus $U=V_{(\Delta,\alpha)}$.
\end{proof}

\section{Indecomposability of the modules $V_{(\Omega_C,\zeta_C)}$}
In this section $\Lambda$ denotes a fixed $k$-graph and $C$ a connected component of $\RG(\Lambda)$. $G$ denotes the fundamental group $\pi(\Gamma_C,y)$ at some fixed vertex $y\in \Gamma_C^{\ob}$, and $KG$ the group algebra of $G$ over $K$. Recall that for a ring $R$, an $R$-module is called \emph{indecomposable} if it is nonzero and cannot be written as a direct sum of two nonzero submodules. It is easy to see that an $R$-module $M$ is indecomposable if and only if $\End_R(M)$ has no {\it nontrivial} idempotents, i.e. idempotents distinct from $0$ and $1$.  We will show that 
\begin{align}
V_{(\Omega_C,\zeta_C)}\text{ is indecomposable}~\Leftrightarrow~ KG \text{ has no nontrivial idempotents.}
\end{align}
In order to prove (1) we will define a subspace $W\subseteq V_{(\Omega_C,\zeta_C)}$ and a subalgebra $A\subseteq \KP(\Lambda)$ such that $W$ is a left $A$-module with the induced action. We will show that
\vspace{0.3cm}
\begin{equation}
\End_{\KP(\Lambda)}(V_{(\Omega,\zeta)})\text{ has no nontrivial idempotents}~\Leftrightarrow~\End_{\bar A}(W)\text{ has no nontrivial idempotents,}
\end{equation}
\vspace{-0.4cm}
\begin{equation}
W\text{ is free of rank }1 \text{ as an }\bar A\text{-module,}
\end{equation}
\vspace{-0.3cm}
\begin{equation}
\bar A \text{ is isomorphic to KG}
\end{equation}
\\
where $\bar A= A/\Ann(W)$. Clearly (2)-(4) imply (1).

In the following we may write $(\Omega,\zeta)$ instead of $(\Omega_C,\zeta_C)$ and $(\Gamma,\xi)$ instead of $(\Gamma_C,\xi_C)$. Recall that $(\Omega,\zeta)$ was defined as follows. An object $(\Delta,\alpha)$ in $C$ was chosen and $(\Omega,\tau)$ was defined as a universal covering of $\Delta$. $\zeta$ was defined as $\alpha\circ\tau$. By Theorem \ref{thmm2} there is a morphism $(\Delta,\alpha)\to(\Gamma,\xi)$ in $\RG(\Lambda)$. It follows from Lemma \ref{lemrepell} that there is a $k$-graph morphism $\eta:\Omega\to\Gamma$ such that $(\Omega,\eta)$ is a universal covering of $\Gamma$ and $\zeta=\xi\circ \eta$. Hence the diagram
\[\xymatrix{
\Omega\ar[r]^\eta\ar@/^2pc/[rr]^\zeta&\Gamma\ar[r]^\xi&\Lambda
}\]
commutes.

We fix a $y\in \Gamma^{\ob}$ and denote the linear subspace of $V_{(\Omega,\zeta)}$ with basis $\eta^{-1}(y)$ by $W$. Moreover, we denote the subalgebra of $\KP(\Lambda)$ consisting of all $K$-linear combination of elements of $\xi({}_y\!\Walk_y(\Gamma))$ by $A$ (note that $\xi({}_y\!\Walk_y(\Gamma))\subseteq \Walk(\Lambda)$). One checks easily that the action of $\KP(\Lambda)$ on $V_{(\Omega,\zeta)}$ induces an action of $A$ on $W$, making $W$ a left $A$-module. Set $\bar A:=A/\Ann(W)$ and let $~\bar~:A\to \bar A,~a\mapsto \bar a$ be the canonical algebra homomorphism. The action of $A$ on $W$ induces an action of $\bar A$ on $W$ making $W$ a left $\bar A$-module. 

If $p\in \xi({}_y\!\Walk_y(\Gamma))$, then there is a unique $\hat p\in {}_y\!\Walk_y(\Gamma)$ such that $\xi(\hat p)=p$ (the uniqueness follows from Lemma \ref{lemwell}). Since $(\Omega,\eta)$ is a covering of $\Gamma$, there is for any $x\in \eta^{-1}(y)$ a unique $\tilde p_x\in \Walk_x(\Omega)$ such that $\eta(\tilde p_x)=\hat p$. We define a semigroup homomorphism $f:\xi({}_y\!\Walk_y(\Gamma))\to \pi(\Gamma,y)$ by $f(p)=[\hat p]$.

\begin{lemma}\label{lemcous}
Let $x\in \Omega^{\ob}$ and $p,q\in \Walk_x(\Omega)$. Then $r(p)=r(q)$ if and only if $[\eta(p)]=[\eta(q)]$ in $\G(\Gamma)$.
\end{lemma}
\begin{proof}
($\Rightarrow$) Suppose that $r(p)=r(q)$. Then $pq^*\in {}_x\!\Walk_x(\Omega)$. By \cite[Theorem 2.7]{PQR} we have $\eta_*\pi(\Omega,x)=\{[\eta(x)]\}$. It follows that $[\eta(p)\eta(q)^*]=\eta_*[pq^*]=[\eta(x)]$ in $\pi(\Gamma,\eta(x))$ whence $[\eta(p)]=[\eta(q)]$.\\
($\Leftarrow$). Suppose that $[\eta(p)]=[\eta(q)]$ in $\G(\Gamma)$. Since $(\Omega,\eta)$ is a covering of $\Gamma$, the map $\G(\Omega)x\to \G(\Gamma)\eta(x)$ induced by $\eta$ is injective. Hence $[p]=[q]$ in $\G(\Omega)$. Since equivalent walks have the same range, we obtain $r(p)=r(q)$.
\end{proof}

The lemma below follow from Lemma \ref{lemcous}.

\begin{lemma}\label{lemzedp}
Let $p,q\in \xi({}_y\!\Walk_y(\Gamma))$. Then the following are equivalent.
\begin{enumerate}[(i)]
\item $f(p)=f(q)$.
\item $r(\tilde p_x)=r(\tilde q_x)$ for some $x\in\eta^{-1}(y)$.
\item $r(\tilde p_x)=r(\tilde q_x)$ for any $x\in\eta^{-1}(y)$.
\end{enumerate}
\end{lemma}

We set $\Ann(x):=\{a\in A\mid a.x=0\}$ for any $x\in\eta^{-1}(y)$, and $\Ann(W):=\{a\in A\mid a.W=0\}$.

\begin{lemma}\label{lemann}
Let $x\in\eta^{-1}(y)$. Then 
\[\Ann(x)=\{\sum_{p\in\xi({}_y\!\Walk_y(\Gamma))}k_pp\in A\mid \sum_{\substack{p\in \xi({}_y\!\Walk_y(\Gamma)),\\r(\tilde p_x)=x'}}k_p=0 \quad \text{for any } x'\in \eta^{-1}(y)\}\]
\end{lemma}
\begin{proof}
Let $a=\sum_{p\in\xi({}_y\!\Walk_y(\Gamma))}k_pp\in A$. Then, in view of Lemma \ref{lemaction},
\begin{align*}
&\hspace{2.35cm}a\in\Ann(x)&\\
&\hspace{1.45cm}\Leftrightarrow~\sum_{p\in\xi({}_y\!\Walk_y(\Gamma))}k_pp.x=0&
\end{align*}
\begin{align*}
\Leftrightarrow~&\sum_{x'\in\eta^{-1}(y)}(\sum_{\substack{p\in \xi({}_y\!\Walk_y(\Gamma)),\\r(\tilde p_x)=x'}}k_p)x'=0\\
\Leftrightarrow~&\sum_{\substack{p\in \xi({}_y\!\Walk_y(\Gamma)),\\r(\tilde p_x)=x'}}k_p=0 \quad\forall x'\in \eta^{-1}(y).
\end{align*}
\end{proof}

\begin{corollary}\label{corann}
$\Ann(W)=\Ann(x)$ for any $x\in\eta^{-1}(y)$.
\end{corollary}
\begin{proof}
Since $\Ann(W)=\bigcap_{x\in\eta^{-1}(y)}\Ann(x)$, it suffices to show that $\Ann(x_1)=\Ann(x_2)$ for any $x_1,x_2\in \eta^{-1}(y)$. So let $x_1,x_2\in \eta^{-1}(y)$. For any $x\in \eta^{-1}(y)$ set $Y_{x}:= \{p\in \xi({}_y\!\Walk_y(\Gamma)\mid r(\tilde p_{x_1})=x\}$ and $Z_{x}:= \{p\in \xi({}_y\!\Walk_y(\Gamma)\mid r(\tilde p_{x_2})=x\}$. Clearly $\xi({}_y\!\Walk_y(\Gamma))=\bigsqcup_{x\in\eta^{-1}(y)} Y_x=\bigsqcup_{x\in\eta^{-1}(y)} Z_x$. It follows from Lemma \ref{lemzedp} that there is a permutation $\pi\in S(\eta^{-1}(y))$ such that $Y_x=Z_{\pi(x)}$ for any $x\in\eta^{-1}(y)$. Let now $a=\sum_{p\in\xi({}_y\!\Walk_y(\Gamma))}k_pp\in A$. Then, by Lemma \ref{lemann}, 
\begin{align*}
&a\in\Ann(x_1)\\
\Leftrightarrow~&\sum_{p\in Y_x} k_p=0~\forall x\in\eta^{-1}(y)\\
\Leftrightarrow~&\sum_{p\in Z_x} k_p=0~\forall x\in\eta^{-1}(y)\\
\Leftrightarrow~&a\in\Ann(x_2).
\end{align*}
\end{proof}

Recall that if $x,x'\in\Omega^{\ob}$, then $x\sim x'\Leftrightarrow \zeta(\Walk_x(\Omega))=\zeta(\Walk_{x'}(\Omega))$.

\begin{lemma}\label{lemalphsim}
$\eta^{-1}(y)$ is a $\sim$-equivalence class.
\end{lemma}
\begin{proof}
Choose an $x\in \eta^{-1}(y)$. Let $x'\in \Omega^{\ob}$ such that $x'\sim x$. It follows from Lemma \ref{lempath} that $\xi(\Walk_{\eta(x)}(\Gamma))=\zeta(\Walk_{x}(\Omega))=\zeta(\Walk_{x'}(\Omega))=\xi(\Walk_{\eta(x')}(\Gamma))$. Hence $\eta(x)\sim\eta(x')$. It follows that $\eta(x')=\eta(x)=y$ since $\Gamma$ is irreducible.\\
Let now $x'\in \eta^{-1}(y)$. Then $\eta(\Walk_x(\Gamma))=\eta(\Walk_{x'}(\Gamma))$ since $(\Omega,\eta)$ is a covering of $\Gamma$. Hence $\zeta(\Walk_x(\Gamma))=\zeta(\Walk_{x'}(\Gamma))$, i.e. $x\sim x'$. We have shown that $\eta^{-1}(y)=[x]_{\sim}$.
\end{proof}

We are ready to prove (2).

\begin{proposition}\label{propextend}
Any nontrivial idempotent endomorphism in $\End_{\KP(\Lambda)}(V_{(\Omega,\zeta)})$ restricts to a nontrivial idempotent endomorphism in $\End_{\bar A}(W)$. Any nontrivial idempotent endomorphism in $\End_{\bar A}(W)$ extends to a nontrivial idempotent endomorphism in $\End_{\KP(\Lambda)}(V_{(\Omega,\zeta)})$. 
\end{proposition}
\begin{proof}
Let $\epsilon\in \End_{\KP(\Lambda)}(V_{(\Omega,\zeta)})$ be a nontrivial idempotent endomorphism. It follows from Lemmas \ref{lemmodule} and \ref{lemalphsim} that $\epsilon(W)\subseteq W$. Hence $\epsilon|_W\in \End_{\bar A}(W)$. Clearly $\epsilon|_W$ is an idempotent. It remains to show that $\epsilon|_W$ is nontrivial. Since $\epsilon$ is nontrivial, there are $v,w\in \Omega^{ob}$ such that $\epsilon(v)\neq 0$ and $\epsilon(w)\neq w$. Let $x\in\eta^{-1}(y)$ and choose a $p\in \zeta({}_v\!\Walk_x(\Omega))$ and a $q\in \zeta({}_w\!\Walk_x(\Omega))$. Then $\epsilon(v)=\epsilon( p.x)=p.\epsilon(x)$ and $\epsilon(w)=\epsilon( q.x)=q.\epsilon(x)$. It follows that $\epsilon(x)\neq 0, x$. Thus $\epsilon|_W$ is nontrivial.\\
Suppose now that $\epsilon_W\in \End_{\bar A}(W)$ is a nontrivial idempotent endomorphism.  Choose an $x\in\eta^{-1}(y)$ such that $\epsilon_W(x)\neq 0$. Since $\Omega$ is connected, we can choose for any $v\in\Omega^{\ob}$ a $p^v\in {}_{v}\!\Walk_x(\Omega)$. Define an endomorphism  $\epsilon$ of the $K$-vector space $V_{(\Omega,\zeta)}$ by 
\[\epsilon(v)=\zeta(p^v).\epsilon_W(x)~\text{ for any }v\in\Omega^{\ob}.
\]
If $x'\in \eta^{-1}(y)$, then 
\[\epsilon(x')=\zeta(p^{x'}).\epsilon_W(x)=\overline{\zeta(p^{x'})}.\epsilon_W(x)=\epsilon_W(\overline{\zeta(p^{x'})}.x)=\epsilon_W(\zeta(p^{x'}).x)=\epsilon_W(x')\]
since $\zeta(p^{x'})\in\xi({}_y\!\Walk_y(\Gamma))$ and $\epsilon_W$ is $\bar A$-linear. Hence $\epsilon_W$ extends to $\epsilon$. Next we show that $\epsilon$ is $\KP(\Lambda)$-linear. Let $t\in \Walk(\Lambda)$ and $v\in \Omega^{\ob}$. We have to prove that 
\begin{equation}
\epsilon(t.v)=t.\epsilon(v).
\end{equation}
Clearly $\epsilon_W(x)=\sum\limits_{i=1}^nk_i x_i$ for some $n\geq 1$, $k_1,\dots,k_n\in K^\times$ and pairwise distinct $x_1,\dots,x_n\in \eta^{-1}(y)$. By Lemma \ref{lemalphsim} we have 
\begin{equation}
x\sim x_1\sim \dots \sim x_n.
\end{equation}
It follows that for any $u\in \Omega^{\ob}$ and $1\leq i\leq n$ there is a $p^{u}_i\in\Walk_{x_i}(\Omega)$ such that $\zeta(p^{u}_i)=\zeta(p^u)$. Set $v_i:=r(p^{v}_i)$ for any $1\leq i\leq n$. It follows from Lemma \ref{lemwithoutnumber} that 
\begin{equation}
v\sim v_1\sim\dots\sim v_n.
\end{equation}
Clearly 
\begin{equation}
t.\epsilon(v)=t.(\zeta(p^v).\sum\limits_{i=1}^nk_i x_i))=\sum\limits_{i=1}^nk_i (t.v_i).
\end{equation}
\\
{\bf Case 1} Assume that $t.v=0$. It follows from (7) that $t.v_i=0$ for any $1\leq i\leq n$ and hence (5) holds (in view of (8)). \\
\\
{\bf Case 2} Assume now that $t.v\neq 0$. Then there is a $q\in \Walk_v(\Omega)$ such that $\zeta(q)=t$. It follows from (7) that for any $1\leq i\leq n$ there is a $q_i\in \Walk_{v_i}(\Omega)$ such that $\zeta(q_i)=t$. Clearly 
\begin{equation}
\epsilon(t.v)=\epsilon(r(q))=\zeta(p^{r(q)}).\epsilon_W(x)=\zeta(p^{r(q)}).\sum\limits_{i=1}^nk_i x_i=\sum\limits_{i=1}^nk_i r(p^{r(q)}_i)
\end{equation}
and 
\begin{equation}
t.\epsilon(v)\overset{(8)}{=}\sum\limits_{i=1}^nk_i (t.v_i)=\sum\limits_{i=1}^nk_i r(q_i).
\end{equation}
We will show that $r(q_i)=r(p^{r(q)}_i)$ for any $1\leq i\leq n$ which implies (5) in view of (9) and (10). Let $1\leq i\leq n$. Then $\zeta(q_ip_i^v)=\zeta(qp^v)$ and $\zeta(p_i^{r(q)})=\zeta(p^{r(q)})$. By Lemma \ref{lemwell} the map $\xi|_{\Walk_y(\Gamma)}:{\Walk_y(\Gamma)}\to {\Walk(\Lambda)}$ is injective. It follows that $\eta(q_ip_i^v)=\eta(qp^v)$ and $\eta(p_i^{r(q)})=\eta(p^{r(q)})$. Hence, by Lemma \ref{lemcous},
\[\Big(r(qp^v)=r(p^{r(q)})\Big)~\Rightarrow~\Big( [\eta(qp^v)]=[\eta(p^{r(q)})] \Big)~\Rightarrow ~\Big([\eta(q_ip_i^v)]=[\eta(p_i^{r(q)})]\Big)~ \Rightarrow ~\Big(r(q_ip_i^v)=r(p_i^{r(q)})\Big)\]
as desired.\\
\\
We have shown that (5) holds and hence $\epsilon\in \End_{\KP(\Lambda)}(V_{(\Omega,\zeta)})$. Since \[\epsilon(v)=\epsilon(\zeta(p_v).x)=\zeta(p_v).\epsilon(x)=\zeta(p_v).\epsilon^2(x)=\epsilon^2(\zeta(p_v).x)=\epsilon^2(v)\]
for any $v\in\Omega^{\ob}$, $\epsilon$ is an idempotent. Clearly $\epsilon$ is nontrivial since $\epsilon|_W=\epsilon_W$ is nontrivial
\end{proof}

Next we prove (3).
\begin{proposition}\label{propfree}
$W$ is free of rank $1$ as an $\bar A$-module.
\end{proposition}
\begin{proof}
Choose an $x\in\eta^{-1}(y)$. If $x'\in\eta^{-1}(y)$, then there is a $p\in{}_{x'}\Walk_x(\Omega)$ since $\Omega$ is connected. Clearly $\zeta(p)\in \xi({}_y\!\Walk_y(\Gamma))$ and moreover $\overline{\zeta(p)}.x=x'$. Hence $x$ generates the $\bar A$-module $W$. On the other hand, if $\bar a.x=0$ for some $a\in A$, then $a\in \Ann(x)$ (since  $\bar a.x=a.x$) and hence $\bar a=0$ by Corollary \ref{corann}. Thus $\{x\}$ is a basis for the $\bar A$-module $W$.
\end{proof}

Recall that the semigroup homomorphism $f:\xi({}_y\!\Walk_y(\Gamma))\to \pi(\Gamma,y)$ was defined by $f(p)=[\hat p]$. Below we prove (4).

\begin{proposition}\label{propKG}
The algebra $\bar A$ is isomorphic to the group algebra $KG$ where $G=\pi(\Gamma,y)$.
\end{proposition}
\begin{proof}
Define the map
\begin{align*}
F:\bar A&\to KG,\\
\overline{\sum_{p\in\xi({}_y\!\Walk_y(\Gamma))}k_pp}&\mapsto \sum_{p\in\xi({}_y\!\Walk_y(\Gamma))}k_pf(p).
\end{align*}
First we show that $F$ is well-defined. Choose an $x\in\eta^{-1}(y)$. Suppose that $\overline{\sum_{p\in\xi({}_y\!\Walk_y(\Gamma))}k_pp}=\overline{\sum_{p\in\xi({}_y\!\Walk_y(\Gamma))}l_pp}$. Then $\sum_{p\in\xi({}_y\!\Walk_y(\Gamma))}(k_p-l_p)p\in \Ann(x)$ by Corollary \ref{corann}. It follows from Lemma \ref{lemann} that 
\begin{equation}
\sum_{\substack{p\in \xi({}_y\!\Walk_y(\Gamma)),\\r(\tilde p_x)=x'}}(k_p-l_p)=0 \quad \text{for any } x'\in \eta^{-1}(y).
\end{equation}
We have to show that $\sum_{p\in\xi({}_y\!\Walk_y(\Gamma))}k_pf(p)=\sum_{p\in\xi({}_y\!\Walk_y(\Gamma))}l_pf(p)$, i.e.
\begin{equation}
\sum_{\substack{p\in \xi({}_y\!\Walk_y(\Gamma)),\\f(p)=g}}(k_p-l_p)=0 \quad \text{for any } g\in G.\quad\quad\hspace{0.05cm}{}
\end{equation}
But it follows from (11) and Lemma \ref{lemzedp} that (12) holds. Hence $F$ is well-defined. We leave it to the reader to check that $F$ is an algebra homomorphism. It remains to show that $F$ is bijective. Suppose that $F( \overline{\sum_{p\in\xi({}_y\!\Walk_y(\Gamma))}k_pp})=F(\overline{\sum_{p\in\xi({}_y\!\Walk_y(\Gamma))}l_pp})$. Then (12) holds. It follows from (12) and Lemma \ref{lemzedp} that (11) holds. Hence $\overline{\sum_{p\in\xi({}_y\!\Walk_y(\Gamma))}k_pp}=\overline{\sum_{p\in\xi({}_y\!\Walk_y(\Gamma))}l_pp}$ and therefore $F$ is injective. The surjectivity of $F$ follows from the surjectivity of $f$.
\end{proof}

We are now in position to prove the main result of this section, namely (1).
\begin{theorem}\label{thmdecomp}
Let $\Lambda$ be a $k$-graph, $C$ a connected component of $\RG(\Lambda)$ and $y\in \Gamma_C^{\ob}$. Then the $\KP(\Lambda)$-module $V_{(\Omega_C,\zeta_C)}$ is indecomposable if and only if the group algebra $KG$ has no nontrivial idempotents where $G=\pi(\Gamma_C,y)$.
\end{theorem}
\begin{proof}
It follows from Propositions \ref{propextend}, \ref{propfree} and \ref{propKG} that
\begin{align*}
&V_{(\Omega,\zeta)}\text{ is indecomposable }\\
\Leftrightarrow~&\End_{\KP(\Lambda)}(V_{(\Omega,\zeta)})\text{ has no nontrivial idempotents}\\
\Leftrightarrow~&\End_{\bar A}(W)\text{ has no nontrivial idempotents}\\
\Leftrightarrow~&KG\text{ has no nontrivial idempotents}.
\end{align*}
\end{proof}

\begin{corollary}\label{cordecomp}
Let $\Lambda$ be a $1$-graph and $C$ a connected component of $\RG(\Lambda)$. Then the $\KP(\Lambda)$-module $V_{(\Omega_C,\zeta_C)}$ is indecomposable
\end{corollary}
\begin{proof}
Choose a $y\in \Gamma_C^{\ob}$. It is easy to see that fundamental groups of $1$-graphs are free. Hence the group ring $KG$, where $G=\pi(\Gamma_C,y)$, has no zero divisors by \cite[Theorem 12]{higman}. It follows that $KG$ has no nontrivial idempotents and hence $V_{(\Omega_C,\zeta_C)}$ is indecomposable, by Theorem \ref{thmdecomp}.
\end{proof}

\section{Examples}
Let $\Lambda$ be a $k$-graph. Choose $k$ colours $c_1,\dots, c_k$. The {\it skeleton} $S(\Lambda)$ of $\Lambda$ is a $(c_1,\dots,c_k)$-coloured directed graph which is defined as follows. The vertices of $S(\Lambda)$ are the vertices of $\Lambda$. The edges of $S(\Lambda)$ are the paths of degree $z_i~(1\leq i\leq n)$ in $\Lambda$ where $z_i$ is the element of $\N^k$ whose $i$-th coordinate is $1$ and whose other coordinates are $0$. The source and the range map in $S(\Lambda)$ are the restrictions of the source and the range map in $\Lambda$, respectively. An edge $e$ in $S(\Lambda)$ has the colour $c_i$ if $e$ is a path of degree $z_i$ in $\Lambda$.
 
By the factorisation property of $\Lambda$, there is a bijection between the $c_ic_j$-coloured paths of length $2$ and the $c_jc_i$-coloured paths. We may think of these pairs as commutative squares in $S(\Lambda)$. Let $\mathcal{C}(\Lambda)$ denote the collection of all commutative squares in $S(\Lambda)$. By a theorem of Fowler and Sims \cite{fowler-sims}, the $k$-graph $\Lambda$ is determined by $S(\Lambda)$ and $\mathcal{C}(\Lambda)$. 

Any directed graph $S$ determines a $1$-graph $\Lambda$ such that $S(\Lambda)=S$. A $(c_1,c_2)$-coloured directed graph $S$ and a collection $\mathcal{C}$ of commutative squares in $S$ which includes each $c_ic_j$-coloured path exactly once, determine a $2$-graph $\Lambda$ such that $S(\Lambda)=S$ and $\C(\Lambda)=\C$. If $k\geq 3$, then the collection $\C$ has to satisfy an extra associativity condition. For more details see \cite[Section 2.1]{pchr}.

Following \cite{PQR}, we call a $k$-graph $\Delta$ a {\it $k$-tree}, if $\pi(\Delta,v)= \{v\}$ for some (and hence any) vertex $v\in \Delta^{\ob}$.

\begin{lemma}\label{lemktree}
Let $\Lambda$ be a $k$-graph and $(\Delta,\alpha)$ an object of the connected component $C$ of $\RG(\Lambda)$. If $\Delta$ is a $k$-tree, then $(\Delta,\alpha)\cong (\Omega_C,\zeta_C)$.
\end{lemma}
\begin{proof}
By Theorem \ref{thmm1} there is a morphism $\phi:(\Omega_C,\zeta_C)\to(\Delta,\alpha)$ in $\RG(\Lambda)$. By Proposition \ref{propsurj}, $(\Omega_C,\phi)$ is a covering of $\Delta$. But since $\Delta$ is a $k$-tree, any covering of $\Delta$ is isomorphic to $(\Delta,\id_\Delta)$ (recall that for any $v\in \Delta^{\ob}$ there is a 1-1 correspondence between the isomorphism classes of coverings of $\Delta$ and the conjugacy classes of subgroups of $\pi(\Delta,v)$). It follows that $\phi:\Omega_C\to \Delta$ is an isomorphism of $k$-graphs and hence $\phi:(\Omega_C,\zeta_C)\to(\Delta,\alpha)$ is an isomorphism in $\RG(\Lambda)$.
\end{proof}

\begin{example}\label{ex1}
Suppose $\Lambda$ is the $1$-graph with skeleton
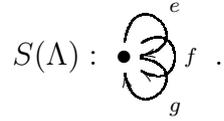
\begin{equation*}
S(\Lambda):\xymatrix{
 \bullet \ar@[blue]@(u,r)^{e} \ar@[blue]@(ur,rd)^{f} \ar@[blue]@(r,d)^{g} &.
}
\end{equation*}
Let $(\Delta,\alpha)$ be the representation $1$-graph for $\Lambda$ whose skeleton is
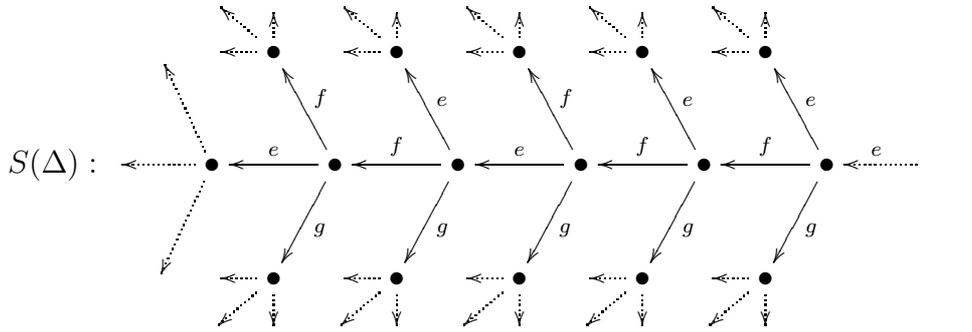
\[
\xymatrix@=10pt{
                   &&         &   &                   &    &        &   &                                   &   &     && &&\\
                & &                      &         \bullet \ar@{.>}[l] \ar@{.>}[u] \ar@{.>}[ul]      &  &  \bullet \ar@{.>}[l] \ar@{.>}[u] \ar@{.>}[ul]&   &     \bullet                 \ar@{.>}[l] \ar@{.>}[u] \ar@{.>}[ul] &  &      \bullet \ar@{.>}[l] \ar@{.>}[u] \ar@{.>}[ul]  && \bullet \ar@{.>}[l] \ar@{.>}[u] \ar@{.>}[ul] &&\\
                   &&                      &&                      &&                       &&                               && &&&&\\
S(\Delta):\, \, \, && \bullet \ar@{.>}[ll]\ar@{.>}[ddl]\ar@{.>}[uul]    &&  \bullet \ar@[blue][ll]_{{e}}\ar@[blue]^{{g}}[ddl] \ar@[blue]_{{f}}[uul] &&   \bullet \ar@[blue][ll]_{{f}} \ar@[blue]^{{g}}[ddl] \ar@[blue]_{{e}}[uul]  && \bullet \ar@[blue][ll]_{{e}} \ar@[blue]^{{g}}[ddl]    \ar@[blue]_{{f}}[uul] && \bullet \ar@[blue][ll]_{{f}}\ar@[blue]^{{g}}[ddl] \ar@[blue]_{{e}}[uul] &&  \bullet\ar@[blue][ll]_{{f}}\ar@[blue]^{{g}}[ddl] \ar@[blue]_{{e}}[uul]  && \ar@{.>}[ll]_{{e}}&.&\\
              &&                      &&                      &&                       &&                               && &&\\
                   & &                &          \bullet \ar@{.>}[l] \ar@{.>}[d] \ar@{.>}[dl]       &   &   \bullet\ar@{.>}[l] \ar@{.>}[d] \ar@{.>}[dl]  &   &    \bullet \ar@{.>}[l] \ar@{.>}[d] \ar@{.>}[dl]                       &    &      \bullet \ar@{.>}[l] \ar@{.>}[d] \ar@{.>}[dl]  &&  \bullet \ar@{.>}[l] \ar@{.>}[d] \ar@{.>}[dl] &&\\
                    &&                      &  &                      &  &             &  &                  & &  &&&&
}
\]
Here the label of an edge indicates its image in $S(\Lambda)$ under $\alpha$. Let $C$ be the connected component of $\RG(\Lambda)$ that contains $(\Delta,\alpha)$.  Since $\Delta$ is a $1$-tree, we have $(\Delta,\alpha)\cong (\Omega_C,\zeta_C)$ by Lemma \ref{lemktree}. For any vertex $v\in\Delta^{\ob}$ and $n\geq 1$ there is precisely one path $p_{v,n}$ of degree($=$length) $n$ ending in $v$. It follows from the irrationality of the infinite path $q=efeffefff...$ that for any distinct vertices $u,v\in \Delta^{\ob}$ there is an $n\geq 1$ such that $\alpha(p_{u,n})\neq \alpha(p_{v,n})$ (cf. \cite{C}, \cite[Section 4]{HPS}). This implies that $\alpha(\Walk_u(\Delta))\neq \alpha(\Walk_v(\Delta))$ for any $u\neq v\in \Delta^{\ob}$, i.e. $(\Delta,\alpha)$ is irreducible. It follows from Theorem \ref{thmm1} that up to isomorphism $(\Delta,\alpha)$ is the only object of $C$. By Theorem \ref{thmirr} the $\KP(\Lambda)$-module $V_{(\Delta,\alpha)}$ is simple. It is isomorphic to the Chen module $V_{[q]}$.
\end{example}

\begin{example}\label{ex2}
Suppose again that $\Lambda$ is the $1$-graph with skeleton
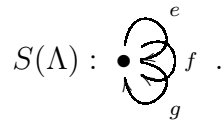
\begin{equation*}
S(\Lambda):\xymatrix{
 \bullet \ar@[blue]@(u,r)^{e} \ar@[blue]@(ur,rd)^{f} \ar@[blue]@(r,d)^{g} &.
}
\end{equation*}
Let $(\Delta_1,\alpha_1)$, $(\Delta_2,\alpha_2)$ and $(\Delta_3,\alpha_3)$ be the representation $1$-graph for $\Lambda$ whose skeletons are
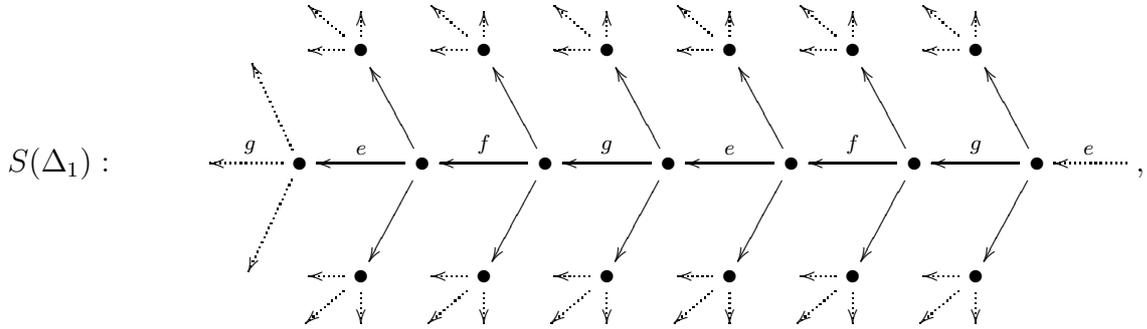
\begin{align*}
&\xymatrix@=10pt{
                  && &&         &   &                   &    &        &   &                                   &&&   &     && &&\\
             && &&   &\bullet \ar@{.>}[l] \ar@{.>}[u] \ar@{.>}[ul] &                      &         \bullet \ar@{.>}[l] \ar@{.>}[u] \ar@{.>}[ul]      &  &  \bullet \ar@{.>}[l] \ar@{.>}[u] \ar@{.>}[ul]&   &     \bullet                 \ar@{.>}[l] \ar@{.>}[u] \ar@{.>}[ul] &  &      \bullet \ar@{.>}[l] \ar@{.>}[u] \ar@{.>}[ul]  && \bullet \ar@{.>}[l] \ar@{.>}[u] \ar@{.>}[ul] &&\\&&
&&                   &&                      &&                      &&                       &&                               && &&&&\\
S(\Delta_1):&&&&\bullet \ar@{.>}[ll]_g\ar@{.>}[ddl]\ar@{.>}[uul] &&  \bullet \ar@[blue][ll]_{{e}} \ar@[blue][ddl]    \ar@[blue][uul]   &&  \bullet \ar@[blue][ll]_{{f}}\ar@[blue][ddl] \ar@[blue][uul] &&   \bullet \ar@[blue][ll]_{{g}} \ar@[blue][ddl] \ar@[blue][uul]  && \bullet \ar@[blue][ll]_{{e}} \ar@[blue][ddl]    \ar@[blue][uul] && \bullet \ar@[blue][ll]_{{f}}\ar@[blue][ddl] \ar@[blue][uul] &&  \bullet\ar@[blue][ll]_{{g}}\ar@[blue][ddl] \ar@[blue][uul]  && \ar@{.>}[ll]_{{e}},&&\\
  &&  &&          &&                      &&                      &&                       &&                               && &&\\
   &&  &&              &\bullet \ar@{.>}[l] \ar@{.>}[d] \ar@{.>}[dl]  &                &          \bullet \ar@{.>}[l] \ar@{.>}[d] \ar@{.>}[dl]       &   &   \bullet\ar@{.>}[l] \ar@{.>}[d] \ar@{.>}[dl]  &   &    \bullet \ar@{.>}[l] \ar@{.>}[d] \ar@{.>}[dl]                       &    &      \bullet \ar@{.>}[l] \ar@{.>}[d] \ar@{.>}[dl]  &&  \bullet \ar@{.>}[l] \ar@{.>}[d] \ar@{.>}[dl] &&\\
    && &&               &&                      &  &                      &  &             &  &                  & &  &&&&
}
\\&
\\
&
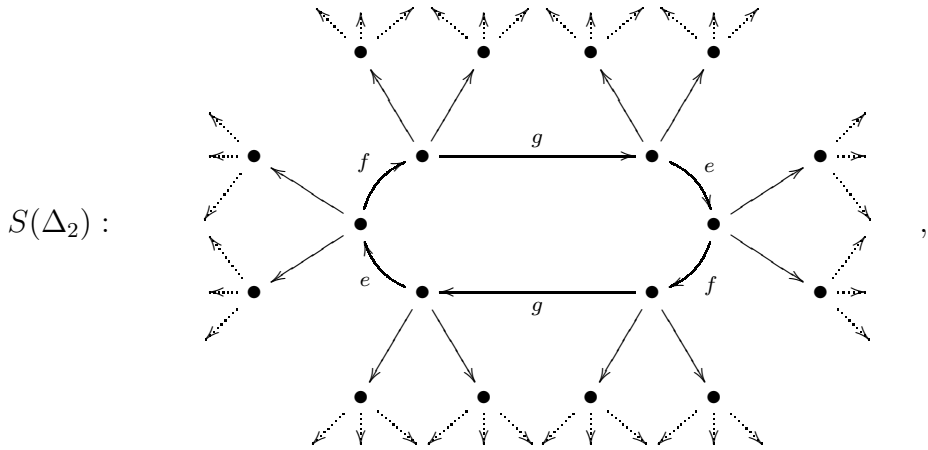
\xymatrix@=10pt{
& && & &  & & & & &&&&&&&\\
& &  & & &\bullet \ar@{.>}[u] \ar@{.>}[ul] \ar@{.>}[ur]  \ar@[blue]@{<-}[ddr]& &\bullet \ar@{.>}[u] \ar@{.>}[ul] \ar@{.>}[ur] \ar@[blue]@{<-}[ddl]&&\bullet \ar@{.>}[u] \ar@{.>}[ul] \ar@{.>}[ur]  \ar@[blue]@{<-}[ddr]&&\bullet \ar@{.>}[u] \ar@{.>}[ul] \ar@{.>}[ur]  \ar@[blue]@{<-}[ddl]&&&  \\
& & & & &  &&&& & & & &&&&\\
& & & \bullet \ar@{.>}[l] \ar@{.>}[ul] \ar@{.>}[dl] \ar@[blue]@{<-}[drr] &&& \bullet\ar@[blue][rrrr] ^{g}&&&&\bullet\ar@[blue]@/^/[dr]^{e} &&&\bullet\ar@{.>}[r] \ar@{.>}[ur]  \ar@{.>}[dr] \ar@[blue]@{<-}[lld]& &&& &  \\
S(\Delta_2):& & & & & \bullet\ar@[blue]@/^/[ur]^{f} &&&&&&\bullet\ar@[blue]@/^/[dl]^{f}&&& &, & & & & & &\\
& & & \bullet \ar@{.>}[l] \ar@{.>}[ul] \ar@{.>}[dl] \ar@[blue]@{<-}[urr]&&&\bullet\ar@[blue]@/^/[ul]^{e}&&&&\bullet\ar@[blue][llll]^{g}&& &\bullet \ar@{.>}[r] \ar@{.>}[ur] \ar@{.>}[dr]   \ar@[blue]@{<-}[llu]& &&&  \\
& & & & &  & & & &&&&&&& &\\
& &  & & &\bullet \ar@{.>}[d] \ar@{.>}[dl] \ar@{.>}[dr]  \ar@[blue]@{<-}[uur] & &\bullet \ar@{.>}[d] \ar@{.>}[dl] \ar@{.>}[dr] \ar@[blue]@{<-}[uul]&&\bullet \ar@{.>}[d] \ar@{.>}[dl] \ar@{.>}[dr]  \ar@[blue]@{<-}[uur]&&\bullet \ar@{.>}[d] \ar@{.>}[dl] \ar@{.>}[dr]  \ar@[blue]@{<-}[uul]&&&  \\
& & & & &  & & & &&&&&&& &\\
}
\\&
\\
&
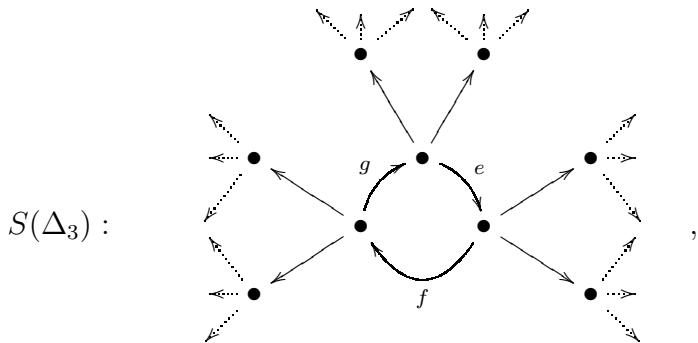
\xymatrix@=10pt{
& && & &  & & & & &\\
& &  & & &\bullet \ar@{.>}[u] \ar@{.>}[ul] \ar@{.>}[ur]  \ar@[blue]@{<-}[ddr] & &\bullet \ar@{.>}[u] \ar@{.>}[ul] \ar@{.>}[ur] \ar@[blue]@{<-}[ddl]&  \\
& & & & &  & & & & &\\
& & & \bullet \ar@{.>}[l] \ar@{.>}[ul] \ar@{.>}[dl] \ar@[blue]@{<-}[drr] &&& \bullet \ar@[blue]@/^/[dr]^{e} &&&\bullet\ar@{.>}[r] \ar@{.>}[ur]  \ar@{.>}[dr] \ar@[blue]@{<-}[lld]& & &  \\
S(\Delta_3):& & & & & \bullet\ar@[blue]@/^/[ur]^{g} &&\bullet \ar@[blue]@/^1.7pc/[ll]^{f}  & & & &, & & & &\\
& & & \bullet \ar@{.>}[l] \ar@{.>}[ul] \ar@{.>}[dl] \ar@[blue]@{<-}[urr]&&&&&& \bullet \ar@{.>}[r] \ar@{.>}[ur] \ar@{.>}[dr]   \ar@[blue]@{<-}[llu]& &  \\
& & & & &  & & & & &\\
}
\end{align*}
respectively. Note that $(\Delta_2,\alpha_2)$ is a quotient of $(\Delta_1,\alpha_1)$, and $(\Delta_3,\alpha_3)$ is a quotient of $(\Delta_2,\alpha_2)$. Let $C$ be the connected component of $\RG(\Lambda)$ that contains $(\Delta_1,\alpha_1)$, $(\Delta_2,\alpha_2)$ and $(\Delta_3,\alpha_3)$. Then $(\Delta_1,\alpha_1)\cong (\Omega_C,\zeta_C)$ since $\Delta_1$ is a $1$-tree, and $(\Delta_3,\alpha_3)\cong (\Gamma_C,\xi_C)$ since $(\Delta_3,\alpha_3)$ is irreducible. By Corollary \ref{cordecomp} the module $V_{(\Delta_1,\alpha_1)}$ is indecomposable and by Theorem \ref{thmirr} the module $V_{(\Delta_3,\alpha_3)}$ is simple. $V_{(\Delta_3,\alpha_3)}$ is isomorphic to the Chen module $V_{[q]}$ where $q$ is the rational infinite path $efgefg...$. 
\end{example}

\begin{example}\label{ex3}
Suppose that $\Lambda$ is the $2$-graph with skeleton
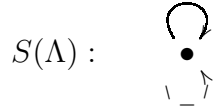
\begin{align*}
\\
\xymatrix{S(\Lambda):&
\bullet \ar@(dl,dr)@{-->}@[red] \ar@(ul,ur)@[blue] &
}\\
\end{align*}
where the solid arrow is blue and the dashed one is red.
Let $(\Delta_1,\alpha_1)$, $(\Delta_2,\alpha_2)$ and $(\Delta_3,\alpha_3)$ be the representation $2$-graph for $\Lambda$ whose skeletons are
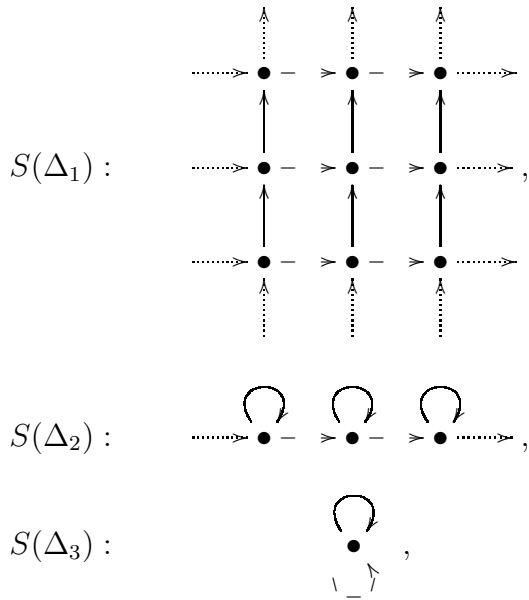
\begin{align*}
&\xymatrix@C=20pt@R=20pt{
&&&&&\\
&\ar@{..>}[r]&\bullet\ar@{-->}@[red][r]\ar@{..>}[u]&\bullet\ar@{-->}@[red][r]\ar@{..>}[u]&\bullet\ar@{..>}[r]\ar@{..>}[u]&\\
S(\Delta_1):&\ar@{..>}[r]&\bullet\ar@{-->}@[red][r]\ar@[blue][u]&\bullet\ar@{-->}@[red][r]\ar@[blue][u]&\bullet\ar@[blue][u]\ar@{..>}[r]&,\\
&\ar@{..>}[r]&\bullet\ar@{-->}@[red][r]\ar@[blue][u]&\bullet\ar@{-->}@[red][r]\ar@[blue][u]&\bullet\ar@[blue][u]\ar@{..>}[r]&\\
&&\ar@{..>}[u]&\ar@{..>}[u]&\ar@{..>}[u]&
}
\\&
\\
&
\xymatrix@C=20pt@R=20pt{
S(\Delta_2):&\ar@{..>}[r]&\bullet\ar@[blue]@(ul,ur)\ar@{-->}@[red][r]&\bullet\ar@[blue]@(ul,ur)\ar@{-->}@[red][r]&\bullet\ar@[blue]@(ul,ur)\ar@{..>}[r]&,}
\\&
\\
&\xymatrix{S(\Delta_3):&\hspace{0.9cm}&
 \bullet \ar@(dl,dr)@{-->}@[red] \ar@(ul,ur)@[blue] &\hspace{-1cm},
}\\
\\
\end{align*}
respectively. Note that $(\Delta_2,\alpha_2)$ is a quotient of $(\Delta_1,\alpha_1)$, and $(\Delta_3,\alpha_3)$ is a quotient of $(\Delta_2,\alpha_2)$. Let $C$ be the connected component of $\RG(\Lambda)$ that contains $(\Delta_1,\alpha_1)$, $(\Delta_2,\alpha_2)$ and $(\Delta_3,\alpha_3)$ (actually it follows from Proposition \ref{prop41.5} that $C$ is the only connected component of $\RG(\Lambda)$). Then $(\Delta_1,\alpha_1)\cong (\Omega_C,\zeta_C)$ since $\Delta_1$ is a $2$-tree, and $(\Delta_3,\alpha_3)\cong (\Gamma_C,\xi_C)$ since $(\Delta_3,\alpha_3)$ is irreducible. Clearly $G:=\pi(\Delta_3,\bullet)$ is the free abelian group on two generators. Hence the group ring $KG$ has no zero divisors by \cite[Theorem 12]{higman}. By Theorem \ref{thmdecomp} the module $V_{(\Delta_1,\alpha_1)}$ is indecomposable and by Theorem \ref{thmirr} the module $V_{(\Delta_3,\alpha_3)}$ is simple. 
\end{example}

\begin{example}\label{ex4}
Suppose $\Lambda$ is the $2$-graph with skeleton
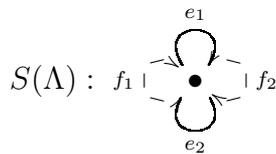
\begin{equation*}
S(\Lambda):\xymatrix{
 \bullet \ar@[blue]@(ul,ur)^{e_1} \ar@[blue]@(dr,dl)^{e_2} \ar@{-->}@[red]@(dl,ul)^{f_1}\ar@{-->}@[red]@(ur,dr)^{f_2}&
}
\end{equation*}
and commutative squares $\C(\Lambda)=\{(e_if_j,f_ie_j)\mid i,j=1,2\}$.
Let $(\Delta,\alpha)$ be the representation $2$-graph for $\Lambda$ whose skeleton is
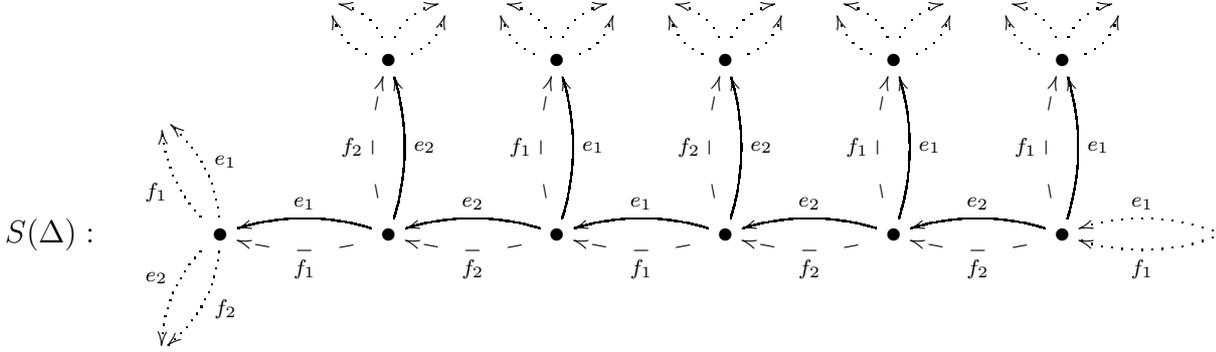
\[
\xymatrix@=12pt{
                   &&         &   &        &&&&&&           &    &        &   &                                   &   &     && &&\\
                & &                      &   & &     \bullet \ar@{.>}@/^/[ul]\ar@{.>}@/_/[ul]\ar@{.>}@/^/[ur]\ar@{.>}@/_/[ur]      & & &   \bullet \ar@{.>}@/^/[ul]\ar@{.>}@/_/[ul]\ar@{.>}@/^/[ur]\ar@{.>}@/_/[ur]   & &  &      \bullet \ar@{.>}@/^/[ul]\ar@{.>}@/_/[ul]\ar@{.>}@/^/[ur]\ar@{.>}@/_/[ur]   & & &       \bullet \ar@{.>}@/^/[ul]\ar@{.>}@/_/[ul]\ar@{.>}@/^/[ur]\ar@{.>}@/_/[ur]    &&&  \bullet \ar@{.>}@/^/[ul]\ar@{.>}@/_/[ul]\ar@{.>}@/^/[ur]\ar@{.>}@/_/[ur]    &&\\
                   &&         &&&&&&             &&                      &&                       &&                               && &&&&\\                   
                   &&         &&&&&&             &&                      &&                       &&                               && &&&&\\
S(\Delta):\, \, \, && \bullet \ar@{.>}@/^/[ddl]^{f_2}\ar@{.>}@/_/[ddl]_{e_2}\ar@{.>}@/^/[uul]^{f_1}\ar@{.>}@/_/[uul]_{e_1}  &&&  \bullet \ar@[blue]@/_/[lll]_{{e_1}}\ar@{-->}@[red]@/^/[lll]^{{f_1}}\ar@[blue]_{{e_2}}@/_/[uuu] \ar@{-->}@[red]^{{f_2}}@/^/[uuu]&&&  \bullet \ar@[blue]@/_/[lll]_{{e_2}}\ar@{-->}@[red]@/^/[lll]^{{f_2}}\ar@[blue]_{{e_1}}@/_/[uuu] \ar@{-->}@[red]^{{f_1}}@/^/[uuu] &&& \bullet \ar@[blue]@/_/[lll]_{{e_1}}\ar@{-->}@[red]@/^/[lll]^{{f_1}}\ar@[blue]_{{e_2}}@/_/[uuu] \ar@{-->}@[red]^{{f_2}}@/^/[uuu] &&& \bullet \ar@[blue]@/_/[lll]_{{e_2}}\ar@{-->}@[red]@/^/[lll]^{{f_2}}\ar@[blue]_{{e_1}}@/_/[uuu] \ar@{-->}@[red]^{{f_1}}@/^/[uuu]
&&& \bullet \ar@[blue]@/_/[lll]_{{e_2}}\ar@{-->}@[red]@/^/[lll]^{{f_2}}\ar@[blue]_{{e_1}}@/_/[uuu] \ar@{-->}@[red]^{{f_1}}@/^/[uuu] &&& \ar@{.>}@/_/[lll]_{{e_1}}\ar@{.>}@/^/[lll]^{{f_1}}&&\\
              &&                    &&&&&&  &&                      &&                       &&                               && &&\\
                   & &          &&&&&&      &              &   &  &   &            &    &      &&   &&
}
\]
(the commutative squares of $S(\Delta)$ are determined by $\alpha$). For any vertex $v\in\Delta^{\ob}$ and $n\geq 1$ there is precisely one path $p_{v,n}$ of degree $(n,0)$ ending in $v$. It follows from the irrationality of $e_1e_2e_1e_2e_2...$ that for any distinct vertices $u,v\in \Delta^{\ob}$ there is an $n\geq 1$ such that $\alpha(p_{u,n})\neq \alpha(p_{v,n})$. This implies that $\alpha(\Walk_u(\Delta))\neq \alpha(\Walk_v(\Delta))$ for any $u\neq v\in \Delta^{\ob}$, i.e. $(\Delta,\alpha)$ is irreducible. Hence, by Theorem \ref{thmirr}, the $\KP(\Lambda)$-module $V_{(\Delta,\alpha)}$ is simple.
\end{example}

\begin{example}\label{ex5}
Suppose again that $\Lambda$ is the $2$-graph with skeleton
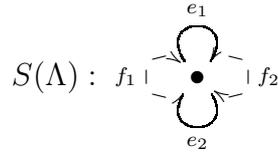
\begin{equation*}
S(\Lambda):\xymatrix{
 \bullet \ar@[blue]@(ul,ur)^{e_1} \ar@[blue]@(dr,dl)^{e_2} \ar@{-->}@[red]@(dl,ul)^{f_1}\ar@{-->}@[red]@(ur,dr)^{f_2}&
}
\end{equation*}
and commutative squares $\C(\Lambda)=\{(e_if_j,f_ie_j)\mid i,j=1,2\}$.
Let $(\Delta_1,\alpha_1)$ and $(\Delta_2,\alpha_2)$ be the representation $2$-graphs for $\Lambda$ whose skeletons are
\vspace{0.4cm}
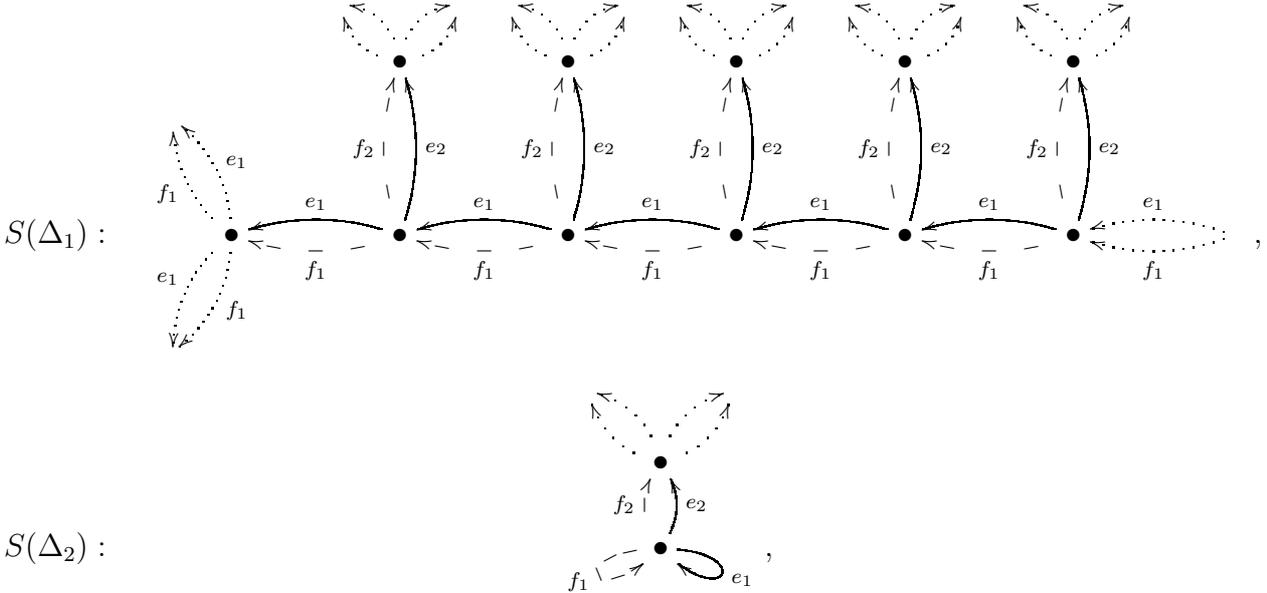
\begin{align*}
&\xymatrix@=12pt{
                   &&         &   &        &&&&&&           &    &        &   &                                   &   &     && &&\\
                & &                      &   & &     \bullet \ar@{.>}@/^/[ul]\ar@{.>}@/_/[ul]\ar@{.>}@/^/[ur]\ar@{.>}@/_/[ur]      & & &   \bullet \ar@{.>}@/^/[ul]\ar@{.>}@/_/[ul]\ar@{.>}@/^/[ur]\ar@{.>}@/_/[ur]   & &  &      \bullet \ar@{.>}@/^/[ul]\ar@{.>}@/_/[ul]\ar@{.>}@/^/[ur]\ar@{.>}@/_/[ur]   & & &       \bullet \ar@{.>}@/^/[ul]\ar@{.>}@/_/[ul]\ar@{.>}@/^/[ur]\ar@{.>}@/_/[ur]    &&&  \bullet \ar@{.>}@/^/[ul]\ar@{.>}@/_/[ul]\ar@{.>}@/^/[ur]\ar@{.>}@/_/[ur]    &&\\
                   &&         &&&&&&             &&                      &&                       &&                               && &&&&\\                   
                   &&         &&&&&&             &&                      &&                       &&                               && &&&&\\
S(\Delta_1):\, \, \, && \bullet \ar@{.>}@/^/[ddl]^{f_1}\ar@{.>}@/_/[ddl]_{e_1}\ar@{.>}@/^/[uul]^{f_1}\ar@{.>}@/_/[uul]_{e_1}  &&&  \bullet \ar@[blue]@/_/[lll]_{{e_1}}\ar@{-->}@[red]@/^/[lll]^{{f_1}}\ar@[blue]_{{e_2}}@/_/[uuu] \ar@{-->}@[red]^{{f_2}}@/^/[uuu]&&&  \bullet \ar@[blue]@/_/[lll]_{{e_1}}\ar@{-->}@[red]@/^/[lll]^{{f_1}}\ar@[blue]_{{e_2}}@/_/[uuu] \ar@{-->}@[red]^{{f_2}}@/^/[uuu] &&& \bullet \ar@[blue]@/_/[lll]_{{e_1}}\ar@{-->}@[red]@/^/[lll]^{{f_1}}\ar@[blue]_{{e_2}}@/_/[uuu] \ar@{-->}@[red]^{{f_2}}@/^/[uuu] &&& \bullet \ar@[blue]@/_/[lll]_{{e_1}}\ar@{-->}@[red]@/^/[lll]^{{f_1}}\ar@[blue]_{{e_2}}@/_/[uuu] \ar@{-->}@[red]^{{f_2}}@/^/[uuu]
&&& \bullet \ar@[blue]@/_/[lll]_{{e_1}}\ar@{-->}@[red]@/^/[lll]^{{f_1}}\ar@[blue]_{{e_2}}@/_/[uuu] \ar@{-->}@[red]^{{f_2}}@/^/[uuu] &&& \ar@{.>}@/_/[lll]_{{e_1}}\ar@{.>}@/^/[lll]^{{f_1}}&\hspace{-0.7cm},&\\
              &&                    &&&&&&  &&                      &&                       &&                               && &&\\
                   & &          &&&&&&      &              &   &  &   &            &    &      &&   &&
}
\\&
\xymatrix@=5pt{
&&&&&&\\
&&&&&&\\
&&&&\bullet \ar@{.>}@/^/[uull]\ar@{.>}@/_/[uull]\ar@{.>}@/^/[uurr]\ar@{.>}@/_/[uurr] &&\\
&&&&&&\\
S(\Delta_2):\hspace{5.5cm}&&&&\bullet\ar@[blue]_{{e_2}}@/_/[uu] \ar@{-->}@[red]^{{f_2}}@/^/[uu]\ar@[blue]@(r,dr)^{{e_1}}\ar@{-->}@[red]@(l,dl)_{{f_1}}&&&,\\
&&&&&&\\
&&&&&&
}
\end{align*}
respectively. Clearly $(\Delta_2,\alpha_2)$ is a quotient of $(\Delta_1,\alpha_1)$. One checks easily that $(\Delta_2,\alpha_2)$ is irreducible. Hence, by Theorem \ref{thmirr}, the module $V_{(\Delta_2,\alpha_2)}$ is simple.
\end{example}

\end{document}